\pdfoutput=1
\documentclass[pdflatex,sn-mathphys-num]{sn-jnl}% Math and Physical Sciences Numbered Reference Style

\usepackage{graphicx}%
\usepackage{multirow}%
\usepackage{amsmath,amssymb,amsfonts}%
\usepackage{amsthm}%
\usepackage{mathrsfs}%
\usepackage[title]{appendix}%
\usepackage{xcolor}%
\usepackage{textcomp}%
\usepackage{manyfoot}%
\usepackage{booktabs}%
\usepackage{algorithm}%
\usepackage{algorithmicx}%
\usepackage{algpseudocode}%
\usepackage{listings}%
\usepackage[font = footnotesize, width = 0.9\textwidth]{caption}%
\newcommand{\cA}{\ensuremath{\mathcal A}}
\newcommand{\cC}{\ensuremath{\mathcal C}}
\newcommand{\cE}{\ensuremath{\mathcal E}}
\newcommand{\cF}{\ensuremath{\mathcal F}}
\newcommand{\cX}{\ensuremath{\mathcal X}}

\newcommand{\bbE}{{\ensuremath{\mathbb E}} }
\newcommand{\bbN}{{\ensuremath{\mathbb N}} }
\newcommand{\bbP}{{\ensuremath{\mathbb P}} }
\newcommand{\bbR}{{\ensuremath{\mathbb R}} }
\newcommand{\bbZ}{{\ensuremath{\mathbb Z}} }

\newcommand{\bfP}{\mathbf{P}}

\newcommand{\1}{{\ensuremath{\mathbf{1}}} }

\newcommand{\pc}{\ensuremath{p_{\mathrm{c}} }}
\newcommand{\pco}{\ensuremath{p_{\mathrm{c}}^{\mathrm{o}} }}

\usepackage{pgf,tikz}
\usetikzlibrary{arrows}
\usetikzlibrary[patterns]
\usetikzlibrary{shapes.misc}
\usetikzlibrary{hobby}

\theoremstyle{thmstyleone}%
\newtheorem{theorem}{Theorem}%  meant for continuous numbers
\newtheorem{proposition}[theorem]{Proposition}% 
\newtheorem{lemma}[theorem]{Lemma}% 
\newtheorem{corollary}[theorem]{Corollary}% 

\theoremstyle{thmstyletwo}%

\theoremstyle{thmstylethree}%

\begin{document}

\title[Catalan percolation]{Catalan percolation}

%%=============================================================%%
%% GivenName	-> \fnm{Joergen W.}
%% Particle	-> \spfx{van der} -> surname prefix
%% FamilyName	-> \sur{Ploeg}
%% Suffix	-> \sfx{IV}
%% \author*[1,2]{\fnm{Joergen W.} \spfx{van der} \sur{Ploeg} 
%%  \sfx{IV}}\email{iauthor@gmail.com}
%%=============================================================%%

\author[1]{\fnm{Eleanor} \sur{Archer}}\email{archer@ceremade.dauphine.fr}
%\equalcont{These authors contributed equally to this work.}

\author*[2]{\fnm{Ivailo} \sur{Hartarsky}}\email{hartarsky@math.univ-lyon1.fr}
%\equalcont{These authors contributed equally to this work.}

\author[3]{\fnm{Brett} \sur{Kolesnik}}\email{brett.kolesnik@warwick.ac.uk}
%\equalcont{These authors contributed equally to this work.}

\author[3]{\fnm{Sam} \sur{Olesker-Taylor}}\email{oleskertaylor.sam@gmail.com}
%\equalcont{These authors contributed equally to this work.}

\author[4]{\fnm{Bruno} \sur{Schapira}}\email{bruno.schapira@univ-amu.fr}
%\equalcont{These authors contributed equally to this work.}

\author[3]{\fnm{Daniel} \sur{Valesin}}\email{daniel.valesin@warwick.ac.uk}
%\equalcont{These authors contributed equally to this work.}

\affil[1]{\orgdiv{Modal'X, UMR CNRS 9023}, 
\orgname{Universit\'e Paris-Nanterre}, 
\orgaddress{\city{Nanterre}, 
\postcode{92000},  
\country{France}}}

\affil[2]{\orgdiv{Institut für Stochastik und Wirtschaftsmathematik}, 
\orgname{Technische Universit\"at Wien}, 
\orgaddress{\street{Wiedner Hauptstra\ss e 8-10}, 
\city{Vienna}, 
\postcode{A-1040},
\country{Austria}}}

\affil[3]{\orgdiv{Department of Statistics}, 
\orgname{University of Warwick}, 
\orgaddress{\city{Coventry}, 
\postcode{CV4 7AL}, 
\country{United Kingdom}}}

\affil[4]{\orgname{Aix-Marseille Universit\'e, CNRS, I2M, UMR 7373}, 
\orgaddress{ 
\city{Marseille}, 
\postcode{13453},  
\country{France}}}

%%%%%%%%%%%%%%%%%%%%%%%%%%%%%%%%%%%%
%%%%%%%%%%%%%%%%%%%%%%%%%%%%%%%%%%%%
%%%%%%%%%%%%%%%%%%%%%%%%%%%%%%%%%%%%
%%%%%%%%%%%%%%%%%%%%%%%%%%%%%%%%%%%%
%%%%%%%%%%%%%%%%%%%%%%%%%%%%%%%%%%%%
%%%%%%%%%%%%%%%%%%%%%%%%%%%%%%%%%%%%
%%%%%%%%%%%%%%%%%%%%%%%%%%%%%%%%%%%%
\abstract{In Catalan percolation, all nearest-neighbour edges $\{i,i+1\}$ along  $\mathbb Z$ are initially occupied, and all other edges are open independently with probability $p$. Open edges $\{i,j\}$ are occupied if some pair of edges $\{i,k\}$ and $\{k,j\}$, with $i<k<j$, become occupied.  This model was introduced by Gravner and the third author, in the context of polluted graph bootstrap percolation.  

We prove that the critical $p_{\mathrm c}$ is strictly between that of oriented site percolation on $\mathbb Z^2$ and the Catalan growth rate $1/4$. Our main result shows that an enhanced oriented percolation model, with non-decaying, infinite-range dependency, has a strictly smaller critical parameter than the classical model. This is reminiscent of the work of Duminil-Copin, Hil\'ario, Kozma and Sidoravicius on brochette percolation. Our proof differs, however, in that we do not use Aizenman--Grimmett enhancements or differential inequalities. Two key ingredients are the work of Hil\'ario, S\'a, Sanchis and Teixeira on stretched lattices, and the Russo--Seymour--Welsh result for oriented percolation by Duminil-Copin, Tassion and Teixeira.}



\maketitle

%%%%%%%%%%%%%%%%%%%%%%%%%%%%%%%%%%%%
%%%%%%%%%%%%%%%%%%%%%%%%%%%%%%%%%%%%
%%%%%%%%%%%%%%%%%%%%%%%%%%%%%%%%%%%%
%%%%%%%%%%%%%%%%%%%%%%%%%%%%%%%%%%%%
%%%%%%%%%%%%%%%%%%%%%%%%%%%%%%%%%%%%
%%%%%%%%%%%%%%%%%%%%%%%%%%%%%%%%%%%%
%%%%%%%%%%%%%%%%%%%%%%%%%%%%%%%%%%%%
\begin{figure}[h]
\includegraphics[width = 0.95\textwidth]{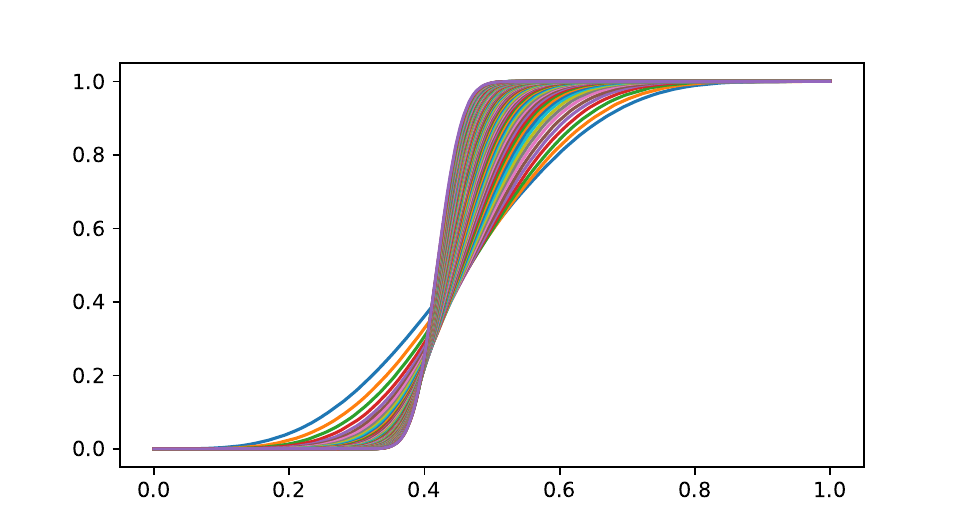}
\caption{Monte Carlo estimates of the conditional probabilities 
that $\{0,n\}$ is occupied given that it is open against $p\in [0,1]$, 
plotted for $n \in \{6, \dots, 100\}$.}
\label{F_sim_stepfn}
\end{figure}

%%%%%%%%%%%%%%%%%%%%%%%%%%%%%%%%%%%%
%%%%%%%%%%%%%%%%%%%%%%%%%%%%%%%%%%%%
%%%%%%%%%%%%%%%%%%%%%%%%%%%%%%%%%%%%
%%%%%%%%%%%%%%%%%%%%%%%%%%%%%%%%%%%%
%%%%%%%%%%%%%%%%%%%%%%%%%%%%%%%%%%%%
%%%%%%%%%%%%%%%%%%%%%%%%%%%%%%%%%%%%
%%%%%%%%%%%%%%%%%%%%%%%%%%%%%%%%%%%%
\section{Introduction}
\label{sec:intro}

%%%%%%%%%%%%%%%%%%%%%%%%%%%%%%%%%%%%
%%%%%%%%%%%%%%%%%%%%%%%%%%%%%%%%%%%%
%%%%%%%%%%%%%%%%%%%%%%%%%%%%%%%%%%%%
%%%%%%%%%%%%%%%%%%%%%%%%%%%%%%%%%%%%
%%%%%%%%%%%%%%%%%%%%%%%%%%%%%%%%%%%%
%%%%%%%%%%%%%%%%%%%%%%%%%%%%%%%%%%%%
%%%%%%%%%%%%%%%%%%%%%%%%%%%%%%%%%%%%

\subsection{Catalan percolation}
\label{subsubsec:Catalan:background}
Catalan percolation stands at the crossroads of 
bootstrap percolation,
oriented percolation and 
enumerative combinatorics.  
It is, in fact, a particular case of the transitive closure dynamics studied 
by Gravner and the third author \cite{Gravner23} (cf.\ Karp \cite{Karp90}
and Kor{\'a}ndi, Peled and Sudakov \cite{Korandi16}).

The original motivation for the model comes from graph bootstrap percolation, 
considered already by Bollob{\'a}s \cite{Bollobas68} (cf.\ 
Balogh, Bollob{\'a}s and Morris \cite{Balogh12b}), 
an early work in the growing field of bootstrap percolation (see, e.g., Morris \cite{Morris17a} for a recent survey). 
More precisely, Catalan percolation is related to 
polluted bootstrap percolation, beginning with    
Gravner and McDonald
\cite{Gravner97a}, 
which amounts to studying bootstrap percolation on a supercritical percolation cluster. 
Roughly speaking, bootstrap percolation is a monotone cellular automaton, modelling the 
spread of ``infection'' in a network. 
Once a site becomes infected, it stays infected thereafter. 
In polluted bootstrap 
percolation, however, some sites are ``immune,'' and so never become infected. 

More specifically, the inspiration for \cite{Gravner23} began with the 
final paragraph in \cite[p.\ 439]{Balogh12b}, which proposes
a polluted version of $H$-bootstrap percolation. Catalan percolation
is associated with the case that $H$ is a directed triangle. 
As is well known, triadic closure plays an important role in, e.g., social networks. 
See, e.g., Granovetter's \cite{Granovetter73} work on ``the strength of weak ties.''
From this point of view, 
Catalan percolation 
(and the transitive closure dynamics, more generally) 
aims to study 
the interplay between the strength of such ties, 
and that of censorship. 
From a combinatorial perspective, as discussed in \cite{Gravner23}, 
$\pc$ for Catalan percolation is also the
point at which a product can be computed at random, when 
brackets are available with probability $p$. 

Let us now formally define the model. 
Fix a parameter $p\in[0,1]$. Consider the complete graph with vertex set $\bbZ$. We start by declaring each edge $\{i,j\}$ with $j\ge i+2$ \emph{open} independently with probability $p$ and \emph{closed} otherwise. We denote this probability measure by $\bfP_p$. We next recursively define a set of \emph{occupied} edges by induction on the length of the edge. Firstly, all edges of the form $\{i,i+1\}$ for $i\in\bbZ$ are occupied. Secondly, each open edge $\{i,k\}$ such that there exists $j\in(i,k)$ such that $\{i,j\}$ and $\{j,k\}$ are both occupied is also occupied, while closed edges cannot be occupied. For $n\ge 2$, we define
\begin{align}
\label{eq:def:theta}\varphi_n(p)&{}=\bfP_p\left(\{0,n\}\text{ is occupied}\mid \{0,n\}\text{ is open}\right),\\
\label{eq:def:pc}\pc&{}=\inf \left\{p: \liminf_{n\to \infty} \varphi_n(p)>0 \right\},
\end{align}
keeping in mind that $\varphi_n(p)$ is monotone in $p$, but not  in $n$. For convenience, we also set $\varphi_1(p)=1/p$ for any $p\in(0,1]$. In view of Figure \ref{F_sim_stepfn}, we expect that $\varphi_n$ converges to the step function $\1_{p>\pc}$, except possibly at $\pc$. Note that in the related oriented percolation setting, this convergence holds also at $\pc$, see 
Bezuidenhout and Grimmett \cite{Bezuidenhout90}.

In \cite{Gravner23} (see Theorem 1.3), it is shown that Catalan percolation has 
a non-trivial phase transition of constant order.  
(On the other hand, for the full transitive closure dynamics, a transition occurs
at $(\log n)^{-1/2+o(1)}$, see Theorems 1.1 and 1.2 in \cite{Gravner23}.) More precisely, using connections with Catalan structures (binary trees) and oriented percolation, 
it can be seen   
(as explained below)  
that 
\begin{equation}\label{eq:GK}1/4\le \pc\le\pco,\end{equation}
where $\pco$ is the critical probability of oriented site percolation on $\bbZ^2$. We refer the reader to Durrett's  classical review \cite{Durrett84} on oriented percolation in two dimensions (see also 
\cite{Liggett99,Liggett05,Hartarsky22GOSP} for more recent and general accounts). For the reader's convenience, we recall that $0.6967\le \pco\le 0.7491$ \cite{Couronne23,Balister94} (also see \cite{Liggett95} for a slightly weaker upper bound). It is believed that $\pco\approx0.7055$ (see, e.g., \cite{Essam88}).

\begin{figure}
\centering
\includegraphics[scale=0.9]{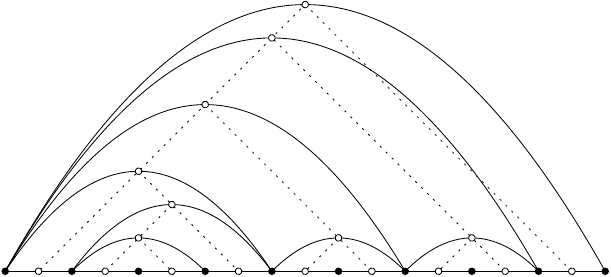}
\caption{Illustration of the binary tree representation of Catalan percolation.}
\label{F_ProdTree}
\end{figure}

The key to \eqref{eq:GK} is the following ``graphical representation'' 
of the Catalan percolation dynamics, 
used in \cite{Gravner23},
from which the connection to binary trees and oriented percolation 
becomes clear. 
For each open or initially occupied edge $\{i,j\}$, with $i<j$, 
place a node $v(i,j)$ at $((i+j)/2,j-i-1)$ in the plane. 
Note that, since all nearest-neighbour edges $\{i,i+1\}$ are initially occupied, 
there are nodes $v(i,i+1)$ at height 0 (i.e., along the $x$-axis) between the integers. 
For all other nodes $v(i,j)$, at some height $j-i-1>0$, we include edges
from $v(i,j)$ to each pair of nodes $v(i,k)$ and $v(k,j)$, with $i<k<j$. 

Clearly, the edge $\{0,n\}$ is occupied by the Catalan percolation 
dynamics if and only if there exists a binary tree rooted at
$v(0,n)$, with leaves $v(0,1),\ldots,v(n-1,n)$. 
See Figure \ref{F_ProdTree}. 
As is well known, the Catalan number 
$C_n=\frac{1}{n+1}\binom{2n}{n}\le 4^n$ counts
the number of 
 such trees. 
Therefore, $p\varphi_n(p)\le 4^np^{n-1}$, leading to the lower bound in \eqref{eq:GK}.

On the other hand, the upper bound in \eqref{eq:GK}
comes from restricting the dynamics in such a way 
that whenever a new edge $\{i,j\}$ is occupied,  
due to some $\{i,k\}$ and $\{k,j\}$, it must be the case
that at least one of $\{i,k\}$ or $\{k,j\}$ is an initially occupied, 
nearest-neighbour edge. 
In other words, the process is forced to ``nucleate,''
in the sense that the maximal length of an occupied edge
can increase by at most 1 in each time step. 
From the perspective of the graphical representation, 
the occupation of $\{0,n\}$, via these restricted
dynamics, 
corresponds 
to 
the presence of an open path from $v(0,n)$ to the $x$-axis 
in oriented site percolation. 
This leads to the upper bound in \eqref{eq:GK}. 
We also note that, from this viewpoint, 
oriented site percolation can be regarded as the local version of 
Catalan percolation, 
in the sense of \cite{Gravner09,Hartarsky24locality}.

The full Catalan percolation dynamics is
richer than either of the two extremes represented in 
\eqref{eq:GK}. 
Indeed, our main result shows 
that $p_c$ lies strictly between the two.

\begin{theorem}
The critical Catalan percolation threshold $\pc$
satisfies 
\[
1/4< \pc<\pco,
\] 
where 
$\pco$ is the critical threshold for oriented site percolation on $\bbZ^2$. 
\end{theorem}

In fact, we will prove a more detailed result, Theorem \ref{th:main} below, 
which requires some additional preparation.

As it is common in percolation (see Grimmett's monograph \cite{Grimmett99} and, e.g., 
the recent work of Duminil-Copin, Goswami, Rodriguez, Severo and Teixeira \cite{Duminil-Copin23a}),
we also introduce critical values of subcritical and supercritical exponential decay, as follows: 
\begin{align}
\label{eq:def:p-}\pc^- &{}= \sup \left\{p: \limsup_{n\to \infty} \frac 1n \log \varphi_n(p)<0 \right\},\\
\label{eq:def:p+}\pc^+&{}=\inf\left\{p:\limsup_{n\to \infty} \frac 1n\log (1-\varphi_n(p)) <0\right\}.
\end{align}
Clearly, $\pc^-\le\pc\le\pc^+$ and it is natural to expect that equality holds, but proving this in a model, 
such as Catalan percolation, 
with such intricate dependencies appears quite challenging. Note that, 
as opposed to more standard percolation models, we have 
above $\pc^+$
that any long open edge is occupied with very high probability. With this notation, the Catalan union bound above actually implies $\pc^-\ge 1/4$. Moreover, in \cite[Section 3]{Gravner23}, a Peierls argument was used to prove that 
\begin{equation}
\label{eq:GK2}\pc^+\le 1-2^{-32}.
\end{equation}

%%%%%%%%%%%%%%%%%%%%%%%%%%%%%%%%%%%%
%%%%%%%%%%%%%%%%%%%%%%%%%%%%%%%%%%%%
%%%%%%%%%%%%%%%%%%%%%%%%%%%%%%%%%%%%
%%%%%%%%%%%%%%%%%%%%%%%%%%%%%%%%%%%%
%%%%%%%%%%%%%%%%%%%%%%%%%%%%%%%%%%%%
%%%%%%%%%%%%%%%%%%%%%%%%%%%%%%%%%%%%
%%%%%%%%%%%%%%%%%%%%%%%%%%%%%%%%%%%%

\subsection{Strict inequalities and stretched lattices in percolation}
In percolation (see \cite{Grimmett99} for background), once the occurrence of a non-trivial phase transition is established, one of the most natural goals is to determine the critical value $\pc$, or its proxies $\pc^-,\pc^+$. 
It is usually not reasonable to expect $\pc$ to have a simple exact expression, so one seeks to estimate or bound this value. It is often the case, as for Catalan percolation, that a simpler reference model (oriented site percolation in our case) can be used to bound the model of interest. If one seeks to improve on the corresponding inequality (the second one in \eqref{eq:GK}), the most classical and, essentially the only, approach is the Aizenman--Grimmett essential enhancement method, as pioneered in \cite{Aizenman91}. Roughly speaking,  
this method 
gives a precise meaning to the intuition that 
if we add a non-trivial amount of connections to the reference model 
(in a way that is not deterministically useless) 
then this strictly decreases the critical parameter.
This is the case when the enhancement is added in an independent way \cite{Aizenman91} (cf.\ 
Balister, Bollob{\'a}s and Riordan \cite{Balister14}). This method has also been influential beyond the realm of percolation (see, e.g., Taggi \cite{Taggi23}).

In the oriented setting, this essential enhancement method fails. Consequently, even simple questions regarding monotonicity of critical values are either still open or the subject of very recent interest. Andjel and Rolla \cite{Andjel23} used a method corresponding to Steps 2 and 3 in Section \ref{sec:outline} below, in order to analyse the effect of boundary enhancement of the one-dimensional contact process. This, mostly classical, part of the argument can be used to tackle independent essential enhancements to oriented models in $1+1$ dimensions. This was indeed implemented for oriented percolation enhanced by diagonal edges by Terra \cite{Terra25}. Strict monotonicity of the critical parameter with respect to dimension was considered by de Lima, Ungaretti and Vares \cite{DeLima24}, using coupling arguments.

However, in models with long range dependency, proving such strict inequalities between critical parameters is much more challenging. Indeed, the only such result we are aware of, for a model with non-decaying correlations, is the work
of Duminil-Copin, Hilário, Kozma and Sidoravicius \cite{Duminil-Copin18b} on brochette percolation and its recent extension to slabs by Castro, Sanchis and Silva \cite{Castro24}. This is achieved by revisiting the Aizenman--Grimmett approach, based on a Russo formula and a partial differential inequality, relating the derivatives of $\varphi_n$ with respect to the parameter $p$ and an enhancement parameter. Yet, the long range of correlations makes the proof quite delicate. In addition to a quantitative version of the essential enhancement idea, \cite{Duminil-Copin18b} relies on refined properties of critical (unoriented) bond percolation on the plane, perhaps the best understood model of percolation \cite{Grimmett99}, as well as a result of Kesten, Sidoravicius and Vares \cite{Kesten22} on oriented percolation in a random environment. In terms of unoriented percolation, \cite{Duminil-Copin18b} uses Russo--Seymour--Welsh results in conjunction with a bound on the 4-arm critical exponent. A further renormalisation leads to oriented percolation in a random environment, for which \cite{Kesten22} establishes that, if the disorder is sufficiently sparse, percolation is maintained.

The result of \cite{Kesten22} is itself highly non-trivial, and should be put in context. It is related to the celebrated work of Hoffman \cite{Hoffman05} on percolation on stretched lattices. While there have been several works investigating what kind of (long-range) disorder destroys percolation,
the recent work of Hil\'ario, S\'a, Sanchis and Teixeira \cite{Hilario23} will be the most relevant 
in our current context. 
In this work, a simplified multi-scale renormalisation 
approach is proposed, for proving that percolation withstands sparse disorder, recovering the results of \cite{Hoffman05,Kesten22}. 
We note that, in \cite{Hilario23}, a certain model of oriented percolation with geometric defects proves instrumental.

%%%%%%%%%%%%%%%%%%%%%%%%%%%%%%%%%%%%
%%%%%%%%%%%%%%%%%%%%%%%%%%%%%%%%%%%%
%%%%%%%%%%%%%%%%%%%%%%%%%%%%%%%%%%%%
%%%%%%%%%%%%%%%%%%%%%%%%%%%%%%%%%%%%
%%%%%%%%%%%%%%%%%%%%%%%%%%%%%%%%%%%%
%%%%%%%%%%%%%%%%%%%%%%%%%%%%%%%%%%%%
%%%%%%%%%%%%%%%%%%%%%%%%%%%%%%%%%%%%
\subsection{Main results}
Our overarching goal, in this work, 
is to further develop tools for proving {\it strict} inequalities for critical 
percolation parameters. We will use Catalan percolation as a study case, 
improving on all of the inequalities in \eqref{eq:GK} and \eqref{eq:GK2}. 
We recall that $\pc^-\le \pc\le\pc^+$, as defined in \eqref{eq:def:pc}, \eqref{eq:def:p-} and \eqref{eq:def:p+}.

\begin{theorem}
\label{th:main}
For Catalan percolation, we have that 
\begin{align}
\label{T_pcLB}
\pc^-&{}> 0.254 
,\\
\label{T_pcUB}\pc^+&{}\le \pco,\\
\label{T_pcp}\pc&{}<\pco.
\end{align}
\end{theorem}

In Section \ref{sec:lower}, we prove 
\eqref{T_pcLB}, via a generating functions approach, which accounts for correlations  
that are omitted in the simple Catalan union bound, discussed above. 

The inequality \eqref{T_pcUB} requires only relatively standard oriented percolation results. 
The short proof of this fact is presented in Section \ref{sec:upper}. 

The proof of \eqref{T_pcp} is the most innovative part of our work. A detailed outline is given in Section \ref{sec:outline} below, 
but let us also make some brief remarks here. 
In Section \ref{sec:main}, we show that, to establish a strict inequality, 
it suffices to introduce only a small amount of the additional 
Catalan percolation dynamics, namely, edges of length two. 
Perhaps the most remarkable feature of our proof is that it does not use 
any form of the Aizenman--Grimmett differential inequality approach to essential enhancements, 
as opposed to \cite{Duminil-Copin18b}. 
We also avoid the use of critical exponent inequalities, which are unavailable in our oriented setting. On the other hand, we still rely on Russo--Seymour--Welsh theory at criticality, which was
 recently established by Duminil-Copin, Tassion and Teixeira \cite{Duminil-Copin17} in the oriented setting, 
 as well as the oriented percolation with geometric defects in \cite{Hilario23}. Curiously, our proof of \eqref{T_pcp} is purely qualitative, and does not yield a quantitative bound. 
 
 While percolation models with strong dependencies are difficult to tackle, we hope that our approach will broaden 
 the scope of models which are amenable to analysis.

%%%%%%%%%%%%%%%%%%%%%%%%%%%%%%%%%%%%
%%%%%%%%%%%%%%%%%%%%%%%%%%%%%%%%%%%%
%%%%%%%%%%%%%%%%%%%%%%%%%%%%%%%%%%%%
%%%%%%%%%%%%%%%%%%%%%%%%%%%%%%%%%%%%
%%%%%%%%%%%%%%%%%%%%%%%%%%%%%%%%%%%%
%%%%%%%%%%%%%%%%%%%%%%%%%%%%%%%%%%%%
%%%%%%%%%%%%%%%%%%%%%%%%%%%%%%%%%%%%
\subsection{Simulations}

We supplement our rigorous results with numerical simulations
in several directions. First, in Figure \ref{F_sim_pcest}, we provide the result of a direct Monte Carlo simulation of the model, determining occupied edges by dynamic programming, 
using the standard increasing coupling of $\bfP_p$ for different values of $p \in [0,1]$.
The results suggest that $\pc \in [0.39, 0.41]$.

\begin{figure}
\centering
\includegraphics[width = 0.8\textwidth]{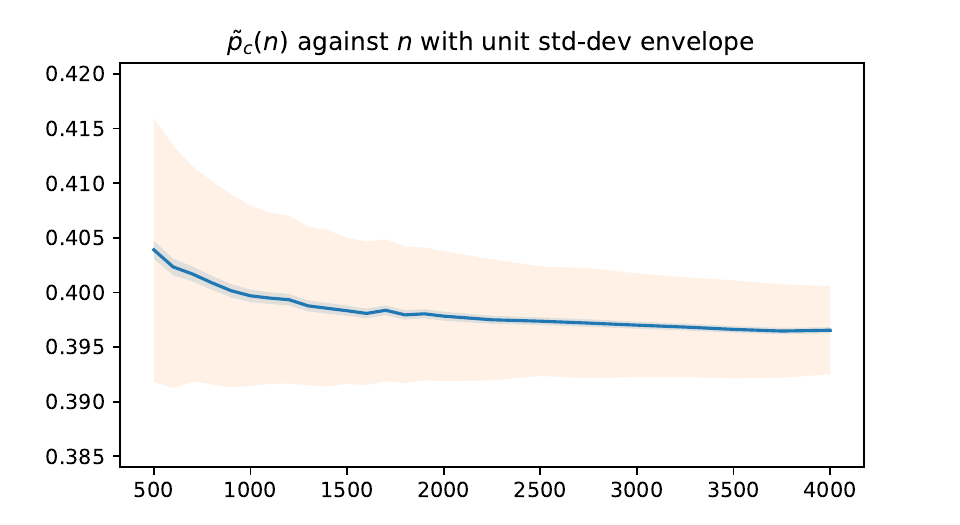}
\caption{We use the standard percolation coupling: edge $\{i,j\}$ is 
assigned an i.i.d.\ $u_{i,j} \sim \operatorname{Unif}(0,1)$, and is open if 
$u_{i,j} \le p$. We condition that $\{0,n\}$ is open. For a given realization, 
we define $\tilde p_\mathrm{c}(n)$ to be the minimal $p$ such that $\{0,n\}$ 
is occupied. The figure plots estimates of the average of $\tilde p_\mathrm{c}(n)$, 
surrounded by a 
one-standard-deviation envelope, estimated via 2000 Monte Carlo rounds.}
\label{F_sim_pcest}
\end{figure}

In Figure \ref{F_sim_upper}, we display a similar Monte Carlo simulation, 
for the Catalan percolation model truncated as in the proof of \eqref{T_pcp}, 
using only edges up to a certain length in the oriented percolation representation. 
The results clearly suggest that the critical values 
of these truncated models converge to $\pc$, as the truncation goes to infinity.

Concerning the lower bound, in Figure \ref{F_sim_lower}, we perform a semi-rigorous study. 
Instead of the exact values of $\varphi_n(p)$ for small $n$, as in the proof of \eqref{T_pcLB}, we use the Monte Carlo estimates of $\varphi_n(p)$, displayed in Figure \ref{F_sim_stepfn}, and plug them into our rigorous lower bound. In this case, the results suggest that our lower bound sequence does {\it not} converge to $\pc$, 
as one takes higher levels of dependency into account. The reasons for this 
are further discussed in
Section \ref{subsec:further} below.

%%%%%%%%%%%%%%%%%%%%%%%%%%%%%%%%%%%%
%%%%%%%%%%%%%%%%%%%%%%%%%%%%%%%%%%%%
%%%%%%%%%%%%%%%%%%%%%%%%%%%%%%%%%%%%
%%%%%%%%%%%%%%%%%%%%%%%%%%%%%%%%%%%%
%%%%%%%%%%%%%%%%%%%%%%%%%%%%%%%%%%%%
%%%%%%%%%%%%%%%%%%%%%%%%%%%%%%%%%%%%
%%%%%%%%%%%%%%%%%%%%%%%%%%%%%%%%%%%%
\subsection{The expected out-degree} 
Let us close 
this introduction 
with some speculation and intrigue. 
Recall that 
$p \varphi_n(p)$ is the probability that $\{0,n\}$ is occupied.
It would appear that the 
expected out-degree of $0$, given 
by the series 
\[
\sum_{n=1}^\infty p\varphi_n(p)
=1+p+2p^2+4p^3+9p^4+21p^5+52p^6+129p^7+335p^8+\cdots,
\]
has positive, integer-valued coefficients. 
If they were to have a combinatorial description, 
then perhaps one could actually locate 
the radius of convergence, and perhaps then 
$p_c$. 

\begin{figure}
\centering
\begin{minipage}{.5\textwidth}
\centering
\includegraphics[width = \textwidth]{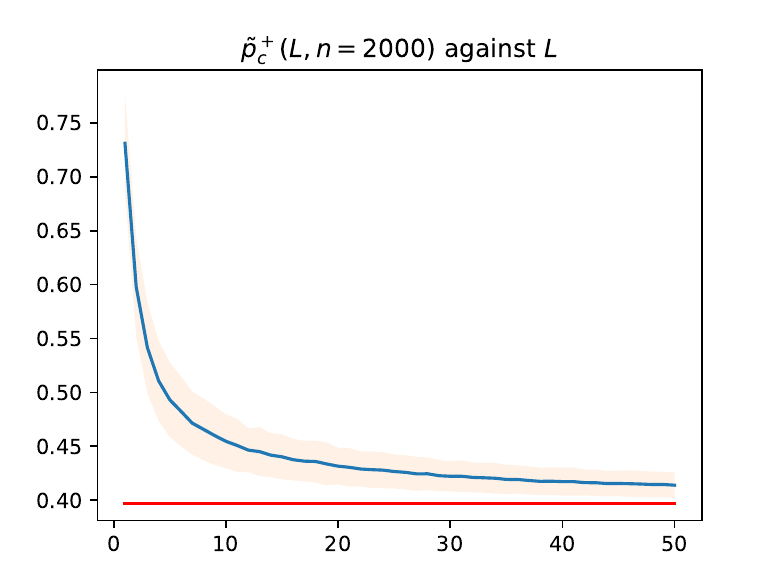}
\captionsetup{width = .9\linewidth}
\caption{We only permit edge $\{i,j\}$ to be occupied via (occupied) 
edges $\{i,k\}$ and $\{k,j\}$ with $|i-k| \le L$ or $|j-k| \le L$. 
We call the resulting threshold $\tilde p_\mathrm{c}^+(L, n)$. 
Clearly, $\tilde p_\mathrm{c}(n) \le \tilde p_\mathrm{c}^+(L, n)$.
We take $n = 2000$, and perform 2000 Monte Carlo estimates, 
and plot (in blue) the mean with a 
one-standard-deviation envelope.
For comparison, we plot (in red) a horizontal line of our 
estimate of $\tilde p_\mathrm{c}(2000) \approx 0.4$.}
\label{F_sim_upper}
\end{minipage}%
\begin{minipage}{.5\textwidth}
\centering
\includegraphics[width = \linewidth]{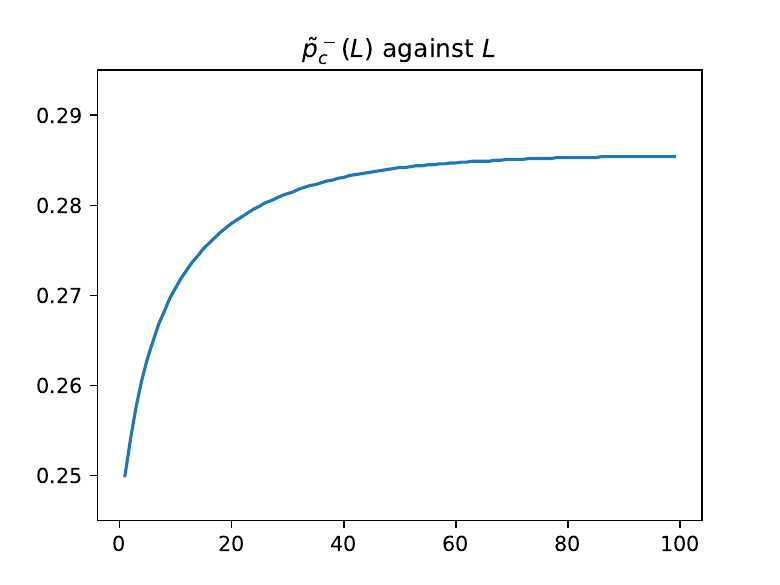}
\captionsetup{width = .9\linewidth}
\caption{We simulate the functions $\varphi_\ell$ for small $\ell \le 100$ 
via $10^6$ Monte Carlo rounds, to precision $10^{-4}$. We plug these 
into our rigorous lower bound developed in Section \ref{sec:lower}: 
the estimate $\tilde p_\mathrm{c}^-(L)$ 
uses these estimates for $\ell\le L$ instead of $\varphi_\ell$. Notice that the curve does \emph{not} seem to converge to $\pc \approx 0.40$.
See Section \ref{subsec:further} for more on this.
For comparison, the real value $\pc^-(1)$ is $0.25$, the Catalan bound.}
\label{F_sim_lower}
\end{minipage}
\end{figure}

%%%%%%%%%%%%%%%%%%%%%%%%%%%%%%%%%%%%
%%%%%%%%%%%%%%%%%%%%%%%%%%%%%%%%%%%%
%%%%%%%%%%%%%%%%%%%%%%%%%%%%%%%%%%%%
%%%%%%%%%%%%%%%%%%%%%%%%%%%%%%%%%%%%
%%%%%%%%%%%%%%%%%%%%%%%%%%%%%%%%%%%%
%%%%%%%%%%%%%%%%%%%%%%%%%%%%%%%%%%%%
%%%%%%%%%%%%%%%%%%%%%%%%%%%%%%%%%%%%

\section{Outline of the proof
that \texorpdfstring{$p_{\mathrm c}<\pco$}{pc<pco}}
\label{sec:outline}

In this section, we discuss the main 
ideas behind the proof that 
$p_{\mathrm c}<\pco$ carried out in detail in Section \ref{sec:main}. 
Several steps are involved, as outlined below.

{\bf Step 1 \rm{(Enhanced oriented percolation)}.}
We first introduce a model of oriented percolation with edges $(x,x+(1,0))$, $(x,x+(0,1))$ and $(x,x+(0,2))$, somewhat similar to the (unoriented) brochette percolation of Duminil-Copin, Hil\'ario, Kozma and Sidoravicius \cite{Duminil-Copin18b}. Sites are open with probability $p$ and length 1 edges are always open. For any $n$, the edges of the form $((x,2n),(x,2n+2))$ are either all closed or all open, the latter having probability $q$. For fixed $q$, we can define a critical value $\pc(q)$. It then suffices to prove that for any $q>0$ we have $\pc(q)<\pc(0)=\pco$. 
Indeed, Catalan percolation with parameter $p$ dominates this enhanced oriented percolation model with $q=p$, so that $\pc\le\max(p,\pc(p))<\pco$ for any $p\in(0,\pco)$. To see this, we consider binary trees such that at each level either one of the children is a leaf, or the second child has exactly two descendants (corresponding to length 2 edges). 

{\bf Step 2 \rm{(Edge speed)}.}
A classical object in 2-dimensional oriented percolation is the edge of the process \cite{Durrett84}. The right (resp.\ left) edge $r_{2n}$ (resp.\ $l_{2n}$) is the largest (resp.\ smallest) $x$ such that $\{\dots,-1,0\}\times \{0\}$ (resp.\ $\{0,1,\dots\}\times\{0\}$) is connected to $(x,2n)$ via an open path. A subadditive theorem of Durrett \cite{Durrett80} (also see \cite{Liggett05}) gives the existence of the right edge speed $\alpha(p,q)=\lim_{n\to\infty}r_{2n}/(2n)$ and similarly for the left edge speed $\beta(p,q)$. It is a classical result of Griffeath \cite{Griffeath81} that $\alpha(\pc(0),0)=\beta(\pc(0),0)=1$. Still by classical means \cite{Durrett80}, we prove that $\alpha$ is strictly increasing 
and $\beta$ strictly decreasing in $q$. While this step requires some minor adaptations, the proofs are essentially identical to the ones for the 
classical 
model with $q=0$. This is 
achieved by 
choosing the correct direction, with respect to which to define the edge speeds, so that dependencies are kept perpendicular to the (vertical) time axis and independence in time is preserved.

{\bf Step 3 \rm{(Crossing good times)}.}
We next show that, whenever $\alpha(p,q)\neq -\infty$, there is a large probability to cross a very elongated parallelogram, whose long side has slope $\alpha(p,q)$ and short side is horizontal, from bottom to top. The proof follows the lines of Durrett \cite{Durrett84} and applies also to $\beta(p,q)$, when $\beta(p,q)\neq +\infty$. We apply this result for some $q>0$ fixed and $p=\pc(0)$, so that $\alpha(p,q)>1>\beta(p,q)$ by Step 2. We call the resulting large parallelogram a (right or left) box. We next view the state of length 2 edges as a random environment. The above yields that there is a high probability ``good'' event on the random environment, on which (vertically) crossing a box is likely.

{\bf Step 4 \rm{(Crossing bad times)}.}
If the environment were always good, we would already be done by constructing a 1-dependent (renormalised) oriented bond percolation out of left and right boxes. However, at some times the environment is bad. Let us focus on an interval of bad times. If the interval is not longer than the height $m$ of a box, we can cross it with high probability via a path of slope 1 by Step 3 applied to $q=0$. However, the bad interval could be much longer. In that case, we still ask for a path of slope (approximately) $1$ with fluctuations of order 
$o(m)$ (see Figure \ref{fig:RSW}). In order to lower bound the probability of such paths, we use the box crossing result of Duminil-Copin, Tassion and Teixeira \cite{Duminil-Copin17} applied at $(p,q)=(\pc(0),0)$. This yields that in an interval of bad times, crossing a rectangle of width $o(m)$ and height $k m$ is at least $\varepsilon^{k}$ for some small $\varepsilon>0$ independent of $m$.

{\bf Step 5 \rm{(Oriented percolation with geometric defects)}.}
With the ingredients above, we renormalise the enhanced oriented percolation model to oriented percolation with geometric defects introduced and studied recently by Hil\'ario, S\'a, Sanchis and Teixeira \cite{Hilario23}, via multi-scale renormalisation. In this model, bonds of the oriented square lattice at ``level'' $i\in\bbZ$ are open independently with probability $p^{1+\xi_i}$, where $\xi_i$ is a sequence of i.i.d.\ geometric random variables. The result of \cite{Hilario23} is that this model percolates if the expectation of the geometric variables is sufficiently low and $p$ is sufficiently close to 1. 

In the renormalisation, edges correspond to boxes at good times, while the 
variables $\xi_i$  encode the lengths of bad time intervals. Indeed, Step 3 ensures that bad times are rare and, at good times, boxes are likely to be crossed, while Step 4 gives that bad intervals are crossed at a cost with an exponential tail, independently of the renormalisation (and therefore independently of how likely the good environment is). Furthermore, the renormalisation is performed carefully, so as to keep crossings of bad times for different renormalised vertices independent (disjoint), which allows renormalised edges to be 1-dependent only at good times. Then, a classical result of Liggett, Schonmann and Stacey \cite{Liggett97} can be used to recover independence. Once this renormalisation is complete, we are able to conclude, because the relevant crossing probabilities are all continuous in $p$, and so we may decrease this parameter a little and remain supercritical.

%%%%%%%%%%%%%%%%%%%%%%%%%%%%%%%%%%%%
%%%%%%%%%%%%%%%%%%%%%%%%%%%%%%%%%%%%
%%%%%%%%%%%%%%%%%%%%%%%%%%%%%%%%%%%%
%%%%%%%%%%%%%%%%%%%%%%%%%%%%%%%%%%%%
%%%%%%%%%%%%%%%%%%%%%%%%%%%%%%%%%%%%
%%%%%%%%%%%%%%%%%%%%%%%%%%%%%%%%%%%%
%%%%%%%%%%%%%%%%%%%%%%%%%%%%%%%%%%%%
\section{Strict lower bound,  \texorpdfstring{$\pc^- > 0.254$}{pc->0.254}}
\label{sec:lower}

First, we will describe our 
general method for lower bounds in Section \ref{ss_method_lower}. In Section \ref{S_Catalan}, for the purpose of illustration, we 
use this method to prove that~$\pc^- \ge 1/4$. Finally, in Section \ref{S_Catalan+}, we push the method further to show that $\pc^- > 0.254$. 

Let~$\theta_n(p) = p \varphi_n(p)$ be the probability that the edge~$\{0,n\}$ is occupied.

%%%%%%%%%%%%%%%%%%%%%%%%%%%%%%%%%%%%
%%%%%%%%%%%%%%%%%%%%%%%%%%%%%%%%%%%%
%%%%%%%%%%%%%%%%%%%%%%%%%%%%%%%%%%%%
%%%%%%%%%%%%%%%%%%%%%%%%%%%%%%%%%%%%
%%%%%%%%%%%%%%%%%%%%%%%%%%%%%%%%%%%%
%%%%%%%%%%%%%%%%%%%%%%%%%%%%%%%%%%%%
%%%%%%%%%%%%%%%%%%%%%%%%%%%%%%%%%%%%

\subsection{Method for lower bound}
\label{ss_method_lower}
Our starting point is expressing~$\pc^-$ in terms of the radius of convergence of a power series. For a sequence~$\{a_n\}$ (with either~$n \ge 0$ or~$n \ge 1$), let~$\mathrm{rad}(\{a_n\}) = 1/\limsup a_n^{1/n}$ denote the radius of convergence of the power series~$\sum_{n} a_nx^n$. Then, recalling the definition of~$\pc^-$ in~\eqref{eq:def:p-}, we have
\begin{equation}\label{eq_pc-radius}
\pc^- = \sup\{p > 0: \mathrm{rad}(\{\theta_n(p)\}) > 1\}.
\end{equation}

Our strategy will be to find 
functions $p \mapsto a_n(p)$, satisfying
\begin{equation}
	\label{eq_an_psi}
	a_n(p) \ge \theta_n(p),\quad p \in [0,1],
\end{equation}
and so that~$\mathrm{rad}(\{a_n(p)\})$ is easy to analyse (by studying the associated generating function). Note that~\eqref{eq_an_psi} gives~$\mathrm{rad}(\{a_n(p)\}) \le \mathrm{rad}(\{\theta_n(p)\})$, so
\begin{equation}
	\label{eq_radius_an}
	\pc^- \ge \sup\{p > 0:\; \mathrm{rad}(\{a_n(p)\}) > 1\}.
\end{equation}

In order to find~$\{a_n(p)\}$ satisfying~\eqref{eq_an_psi}, we will use the recurrence relation
\begin{equation}
	\label{eq_recursive_theta}
	\theta_n(p) \le p \sum_{k=1}^{n-1} \theta_k(p) \theta_{n-k}(p),
\end{equation}
which follows from the definition of an edge being occupied and a union bound. 
More specifically, for
fixed~$n_0 \ge 1$, 
we will define 
$\{a_n^{(n_0)}(p)\}$ by using
the precise probabilities 
$\theta_n(p)$ for small 
$n\le n_0$, and the 
union bound for all larger $n>n_0$. 
Formally, we set
\begin{equation}
	\label{eq_how_to_define_a_n}
	a_n^{(n_0)}(p) = \begin{cases}\theta_n(p),&1\le n \le n_0;\\[.2cm] p\sum_{k=1}^{n-1} a_k^{(n_0)}(p) a^{(n_0)}_{n-k}(p),&n > n_0.\end{cases}
\end{equation}
Comparing this with~\eqref{eq_recursive_theta}, and using induction, it follows  that~\eqref{eq_an_psi} holds.

%%%%%%%%%%%%%%%%%%%%%%%%%%%%%%%%%%%%
%%%%%%%%%%%%%%%%%%%%%%%%%%%%%%%%%%%%
%%%%%%%%%%%%%%%%%%%%%%%%%%%%%%%%%%%%
%%%%%%%%%%%%%%%%%%%%%%%%%%%%%%%%%%%%
%%%%%%%%%%%%%%%%%%%%%%%%%%%%%%%%%%%%
%%%%%%%%%%%%%%%%%%%%%%%%%%%%%%%%%%%%
%%%%%%%%%%%%%%%%%%%%%%%%%%%%%%%%%%%%
\subsection{Catalan bound, revisited}
\label{S_Catalan}
We first implement the above method with~$n_0 = 1$ in~\eqref{eq_how_to_define_a_n}. 

Recall that the Catalan numbers are given by~$C_n = \frac{1}{n+1} \binom{2n}{n}$ for~$n \in \mathbb N$, and satisfy
\begin{equation}\label{Catalan}
C_n = \sum_{k=0}^{n-1} C_k C_{n-k-1}, \quad n \ge 1.
\end{equation}
Noting that~$a^{(1)}_1(p) = C_0 = 1$, comparing~\eqref{eq_how_to_define_a_n} and~\eqref{Catalan} and using induction, we see that
\begin{equation*}
	a^{(1)}_n(p) = p^{n-1}  C_{n-1}, \quad n \ge 1.
\end{equation*}
In particular,~$\mathrm{rad}(\{a^{(1)}_n(p)\}) = \tfrac{1}{p}\mathrm{rad}(\{C_n\})$. It is well known that~$\mathrm{rad}(\{C_n\}) = 1/4$ (this can for instance be checked using~$C_n = \frac{1}{n+1} \binom{2n}{n}$ and Stirling's formula). Hence,~$\mathrm{rad}(\{a^{(1)}_n(p)\}) = 1/(4p)$, and now~$\pc^- \ge 1/4$ readily follows from~\eqref{eq_radius_an}.

%%%%%%%%%%%%%%%%%%%%%%%%%%%%%%%%%%%%
%%%%%%%%%%%%%%%%%%%%%%%%%%%%%%%%%%%%
%%%%%%%%%%%%%%%%%%%%%%%%%%%%%%%%%%%%
%%%%%%%%%%%%%%%%%%%%%%%%%%%%%%%%%%%%
%%%%%%%%%%%%%%%%%%%%%%%%%%%%%%%%%%%%
%%%%%%%%%%%%%%%%%%%%%%%%%%%%%%%%%%%%
%%%%%%%%%%%%%%%%%%%%%%%%%%%%%%%%%%%%
\subsection{Beyond Catalan}
\label{S_Catalan+}
Taking~$n_0 = 2$ in~\eqref{eq_how_to_define_a_n} would not improve on the above, since~$a^{(2)}_2(p) = a^{(1)}_2(p) = p$, and hence~$a^{(2)}_n(p) = a^{(1)}_n(p)$ for all~$n$ and~$p$.
Therefore, we take~$n_0 = 3$. Note that
\begin{equation}\label{eq_a3_values}
	a^{(3)}_1(p) = 1 = a^{(2)}_1(p),\qquad a^{(3)}_2(p)  = p = a^{(2)}_2(p),\qquad a^{(3)}_3(p) = 2p^2 - p^3 < 2p^2=  a^{(2)}_3(p). 
\end{equation}

We now study~$\mathrm{rad}(\{a^{(3)}_n(p)\})$, which we abbreviate as~$x_3(p)$.
Define the power series
\[\cC(x) = \sum_{n = 1}^\infty a^{(3)}_n(p) x^{n},\]
suppressing the dependence on~$p$.
For~$n \ge 4$, we have~$a^{(3)}_n(p)=p\sum_{k=1}^{n-1} a^{(3)}_k(p) a^{(3)}_{n-k}(p)$. 
Multiplying this by~$x^n$, and summing over~$n \ge 4$, gives
\[\cC(x) - a^{(3)}_1(p)  x - a^{(3)}_2(p) x^2 - a^{(3)}_3(p) x^3 = p \left( \cC(x)^2 - (a^{(3)}_1(p))^2  x^2 - 2 a^{(3)}_1(p) a^{(3)}_2(p)  x^3\right).\]
Then, using~\eqref{eq_a3_values} and simplifying, we obtain
\[p \cC(x)^2 - \cC(x) + x - p^3 x^3 = 0.\]
In other words, the quadratic equation
\[ pX^2 - X + x - p^3x^3 = 0\]
is solved by~$X = \cC(x)$ for any~$x < x_3(p)$. The discriminant of this quadratic equation is
\[\Delta (p,x) = 4p^4x^3 - 4px + 1,\]
and (with~$p$ being fixed) the smallest positive value of~$x$ for which~$\Delta (p,x) = 0$ is~$x = x_3(p)$. 
See, e.g., 
Flajolet and Sedgewick \cite[Lemma VII.4]{Flajolet09} for general theoretical background. 

The above considerations imply that the map $p\mapsto x_3(p)$ is continuous on~$(0,1]$, and that 
$x_3(p) \to x_3(0) =\infty$, 
as $p \to 0$.  It then follows that
\begin{equation}
	\label{eq_completion1}
\sup\{p > 0: x_3(p) > 1\} \ge \inf\{p > 0: x_3(p) = 1\}.
\end{equation}
The set of~$p > 0$ for which~$x_3(p) = 1$ is contained in the set of~$p > 0$ for which~$\Delta(p,x_3(p)) = \Delta (p,1)$. 
Therefore, since~$\Delta(p,x_3(p)) = 0$, 
the right-hand side is larger than or equal to
\begin{equation}
	\label{eq_completion2}
 \inf\{p > 0: \Delta(p,x_3(p)) = \Delta(p,1)\}=\inf\{p > 0: \Delta(p,1) = 0\}.
\end{equation} 
Using these considerations, together with~\eqref{eq_radius_an}, we see that~$\pc^-$ is larger than the smallest positive~$p$ satisfying~$4p^4-4p+1 = 0$, which is larger than~$0.254 > 1/4$.

%%%%%%%%%%%%%%%%%%%%%%%%%%%%%%%%%%%%
%%%%%%%%%%%%%%%%%%%%%%%%%%%%%%%%%%%%
%%%%%%%%%%%%%%%%%%%%%%%%%%%%%%%%%%%%
%%%%%%%%%%%%%%%%%%%%%%%%%%%%%%%%%%%%
%%%%%%%%%%%%%%%%%%%%%%%%%%%%%%%%%%%%
%%%%%%%%%%%%%%%%%%%%%%%%%%%%%%%%%%%%
%%%%%%%%%%%%%%%%%%%%%%%%%%%%%%%%%%%%
\subsection{Further iterations}
\label{subsec:further}
Of course, it is possible to
obtain increasingly  better bounds, 
by taking increasingly larger $n_0$ in~\eqref{eq_how_to_define_a_n}. 
Let~$p_m = \sup\{p > 0: \mathrm{rad}(\{a_n^{(m)}(p)\}) > 1\}$, so that, by~\eqref{eq_recursive_theta}, we have~$\pc^- \ge p_m$ for any~$m$.
The sequence $(p_m)_{m\ge1}$ is estimated in Figure \ref{F_sim_lower}. In principle, $\theta_n$ can be written down for arbitrarily large $n$, but it gets ever more complicated. Instead, we used Monte Carlo to estimate $\varphi_m$, and hence $\theta_m$, for $m \le 100$ to obtain Figure \ref{F_sim_lower}. It appears to converge to between $0.28$ and $0.29$, which is much less than our 
numerical estimate $p_c\approx 0.4$.

Roughly speaking, the reason for this 
is that our method accounts only for 
``microscopic'' dependencies. 
That is, even if we plug in the exact
values of $\theta_\ell$, for all $\ell\le n_0$, 
for some large $n_0$, 
into the recursive upper bound 
\eqref{eq_how_to_define_a_n}
on $\theta_n$, we then take $n\to\infty$,
with $n_0$ fixed, in the above analysis. 
As such, this method misses the effect of ``macroscopic'' dependencies. 
For instance, 
note that, crucially, 
it does not account for the fact that, for $n\ge n_0$, the events that $\{0,n\}$ and $\{1,n+1\}$ are occupied are far from being disjoint.

%%%%%%%%%%%%%%%%%%%%%%%%%%%%%%%%%%%%%%%
%%%%%%%%%%%%%%%%%%%%%%%%%%%%%%%%%%%%%%%
%%%%%%%%%%%%%%%%%%%%%%%%%%%%%%%%%%%%%%%
%%%%%%%%%%%%%%%%%%%%%%%%%%%%%%%%%%%%%%%
%%%%%%%%%%%%%%%%%%%%%%%%%%%%%%%%%%%%%%%
%%%%%%%%%%%%%%%%%%%%%%%%%%%%%%%%%%%%%%%

%%%%%%%%%%%%%%%%%%%%%%%%%%%%%%%%%%%%
%%%%%%%%%%%%%%%%%%%%%%%%%%%%%%%%%%%%
%%%%%%%%%%%%%%%%%%%%%%%%%%%%%%%%%%%%
%%%%%%%%%%%%%%%%%%%%%%%%%%%%%%%%%%%%
%%%%%%%%%%%%%%%%%%%%%%%%%%%%%%%%%%%%
%%%%%%%%%%%%%%%%%%%%%%%%%%%%%%%%%%%%
%%%%%%%%%%%%%%%%%%%%%%%%%%%%%%%%%%%%
\section{Upper bound, \texorpdfstring{$\pc^+\le \pco$}{pc+<=pco}}
\label{sec:upper}

Recall $\pc^+$ from \eqref{eq:def:p+}. In this section, 
we show that $\pc^+\le \pco$. 

\subsection{Coupling with oriented percolation}
\label{sec:coupling:OP}
{We start by explaining the coupling with oriented percolation discussed in Section \ref{subsubsec:Catalan:background} in more detail. Let $\bbP_p$ denote the probability measure such that each site $(m,n)\in\bbZ^2$ with $m+n$ even is \emph{open} independently with probability $p$. To define the Catalan percolation configuration, for $j\ge i+2$, we declare the edge $\{i,j\}\subset\bbZ$ \emph{open}, whenever the site $(i+j,|j-i|)$ is open.} 
Note that, we are, for convenience, 
considering a slight modification 
(scaled and translated) 
of the 
coupling in Section \ref{subsubsec:Catalan:background}. 
In particular, we now have that sites at ``level'' $k$ represent edges of length $k$. Let $L_k=\bbZ\times\{k\}$ denote the set of vertices with $y$-coordinate $k$.

For $\ell \leq m$ and $v_1\in L_m$, an \emph{open path} from $v_1$ to $L_\ell$ is a sequence of open sites $v_1, v_2, \ldots v_{m-\ell}$ such that $v_i-v_{i-1}\in\{(-1,-1),(1,-1)\}$ for all $1 < i \leq m$. Note that $v_{m-\ell}\in L_{\ell+1}$, if $m\neq \ell$. We denote by $v\to L_\ell$ the event that there exists an open path from $v$ to $L_\ell$. Open paths therefore correspond to sequences of occupied edges, growing in length one unit at each time step; see Figure \ref{F_triangle_perc}. In particular, if there is an open path from the site $\left( i+j, |j-i| \right)$ to the line $L_1$, this implies that the edge $\{i,j\}$ is occupied in Catalan percolation. 

Finally, we recall the critical threshold of oriented site percolation on $\bbZ^2$:
\[
\pco=\inf\left\{p>0:{\liminf_{n\to\infty} \bbP_p\left((1,1)\to L_{-n} \right)}>0\right\}.
\]

\begin{figure}[h]
\centering
\includegraphics[width = 0.9\textwidth]{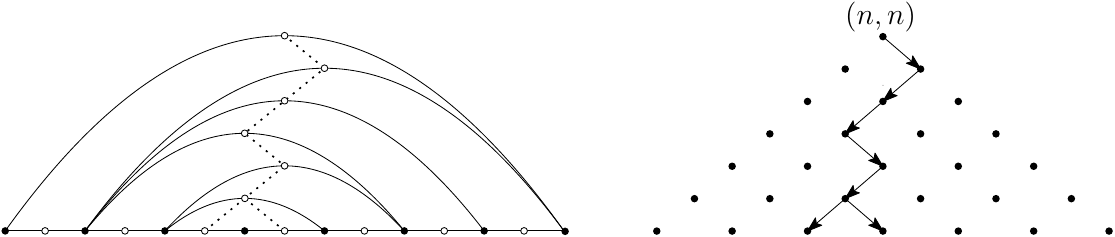}
\caption{An example 
of the oriented percolation coupling, 
with $n=7$.  
\textit{On the left:} a series of occupied edges, 
in which each is obtained by extending the one above it by one 
unit to the left or right. 
\textit{On the right:} the associated path in the oriented site 
percolation model.
Note that the bottom left corner is $(1,1)$ and the bottom 
right corner is $(2n-1,1)$.} 
\label{F_triangle_perc}
\end{figure}

%%%%%%%%%%%%%%%%%%%%%%%%%%%%%%%%%%%%
%%%%%%%%%%%%%%%%%%%%%%%%%%%%%%%%%%%%
%%%%%%%%%%%%%%%%%%%%%%%%%%%%%%%%%%%%
%%%%%%%%%%%%%%%%%%%%%%%%%%%%%%%%%%%%
%%%%%%%%%%%%%%%%%%%%%%%%%%%%%%%%%%%%
%%%%%%%%%%%%%%%%%%%%%%%%%%%%%%%%%%%%
%%%%%%%%%%%%%%%%%%%%%%%%%%%%%%%%%%%%
\subsection{Proof}
We now give the proof of \eqref{T_pcUB}. The strategy is to show that there is a very high probability of finding an integer $k$ such that the edges $\{0,k\}$ and $\{k, n\}$ are both occupied in Catalan percolation. In the coupling with percolation, this corresponds to finding a $k$ such that the vertices $\left( k, k\right)$ and $\left( n+k, n-k\right)$ are both connected to the line $L_1$ by an open path.

Let us note that the coupling, and the
general strategy described above, are, in fact, 
the same as in \cite[Section 3]{Gravner23}. 
However, our current proof 
leads to a stronger result. 
In \cite{Gravner23}, 
\eqref{eq:GK2} is proved using a 
Peierls argument. On the other hand, 
our current proof of
\eqref{T_pcUB}
rests on the following two, classical results from oriented (site) percolation.

\begin{theorem}[Exponential death bound \cite{Durrett83}]
\label{th:death}
For any $p>\pco$, there exists $c>0$ such that,  for any $k\le n$, we have that 
\[\bbP_p\left(\left(1,1\right)\to L_{-k},\left(1,1\right) \not\to L_{-n}\right)\le e^{-ck}.\]
\end{theorem}

\begin{theorem}[Large deviations of the density of the infinite cluster \cite{Durrett88}]
\label{th:LD}
For any $p>\pco$, there exist $\varepsilon,c>0$ such that, for any integer $n \geq 1$ and finite set $A\subset \bbZ_+$, we have that 
\[\bbP_p\left(|\{a\in A:(a,a)\to L_{-n}\}|\le \varepsilon |A|\right)\le e^{-c|A|}.\]
\end{theorem}

Strictly speaking, \cite{Durrett88} proves this result with $\{(a,a):a\in A\}$ replaced by an interval of the form $\{(a,0):a\in\{1,\dots,|A|\}\}$, but the same proof works. As noted in \cite{Durrett88}, the proof applies to oriented percolation, in addition to the contact process.

\begin{proof}[Proof of $\pc^+\le \pco$]
Fix $p>\pco$ and a large enough integer $n\ge 2$. Define the random sets
\begin{align*}
A&=\left\{a\in\bbZ\cap[7n/16,9n/16]: (a,a) \to L_{\lceil 3n/8\rceil}\right\}, \\
B&=\left\{a\in\bbZ\cap[7n/16,9n/16]: {(n+a,n-a)} \to L_{\lceil 3n/8\rceil}\right\}.
\end{align*}
Roughly speaking, these are 
the positions of the sites around the middle of the left and right sides of the triangle in Figure \ref{F_triangle_perc3}, with fairly long open paths to level $3n/8$. By Theorem \ref{th:LD} we have $\bbP_p(|A|<\varepsilon n)\le e^{-cn}$, for suitable $\varepsilon,c>0$, independent of $n$. 

Notice that $A$ and $B$ are measurable with respect to the state of sites in 
\begin{align*}
T&{}=\bigcup_{k=\lceil 3n/8\rceil}^{\lfloor 9n/16\rfloor} \left\{ (k + 2\ell, k): 0 \leq \ell \leq \lfloor 9n/16\rfloor-k\right\},\\
T'&{}=\bigcup_{k=\lceil 3n/8\rceil}^{\lfloor 9n/16\rfloor} \left\{ (2n-k - 2\ell, k): 0 \leq \ell \leq \lfloor 9n/16\rfloor-k\right\}, 
\end{align*}
respectively, and that these two triangles are disjoint. See Figure \ref{F_triangle_perc3}.

\begin{figure}[h]
\centering
\includegraphics[width = 0.85\textwidth]{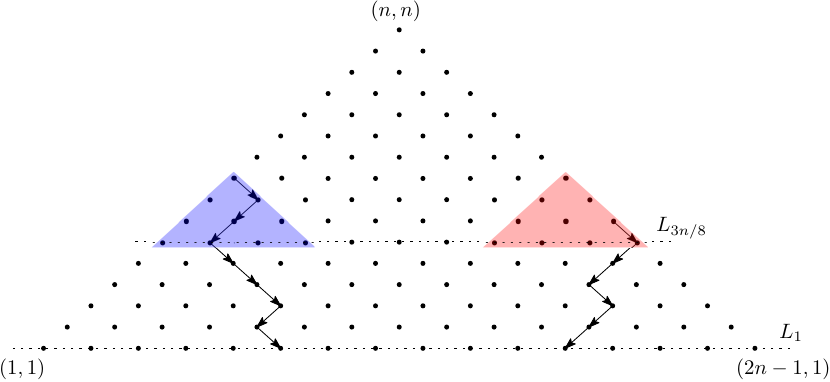}
\caption{An example with $n=16$. Note that 
$A$ and $B$ are measurable with respect to the sets of sites in $T$ and $T'$, 
shaded in blue and red, respectively. 
Here $a=9n/16$ realises the desired event.}\label{F_triangle_perc3}
\end{figure}

Therefore, by independence, symmetry and Theorem \ref{th:LD}, we find that
\[\bbP_p(\not \exists a\in A \cap B|A)\le e^{-c|A|}.\]
Note that, we have, in fact, only used a weaker version of Theorem \ref{th:LD}, going back to \cite{Durrett83} (see also \cite[Section~9]{Durrett84}).

Finally, applying Theorem \ref{th:death} (see again  Figure \ref{F_triangle_perc3}), we obtain
\begin{multline*}
\bfP_p(\{0,n\}\text{ is open, but not occupied})\\
\begin{aligned}
\le{}& p\bfP_p\left(\not\exists a\in\bbZ\cap[7n/16,9n/16]: \{0,a\} \text{ and } \{a,n\} \text{ {occupied}}\right)\\
\le{}& p\bbP_p\left(\not\exists a\in\bbZ\cap[7n/16,9n/16]:(a,a)\to L_1,(n+a,n-a)\to L_1\right)\\
\le{}& \bbP_p\left(|A|<\varepsilon n\right)+\bbP_{p}\left(|A|\ge \varepsilon n,\not\exists a\in A \cap B \right) \\
&{}+ \bbP_p\left(\exists (i,j) \in T\cup T'
: j \geq \tfrac{7n}{16}, (i,j) \to L_{\lceil 3n/8\rceil}, (i,j) \not\to L_1\right)\\
\le{}& e^{-cn}+e^{-c\varepsilon n}+n^2e^{-c(\lfloor n/16\rfloor-1)}.\end{aligned}\end{multline*}
Since $n$ can be taken arbitrarily large, with $c,\varepsilon>0$ fixed, this 
concludes the proof.
\end{proof}

%%%%%%%%%%%%%%%%%%%%%%%%%%%%%%%%%%%%
%%%%%%%%%%%%%%%%%%%%%%%%%%%%%%%%%%%%
%%%%%%%%%%%%%%%%%%%%%%%%%%%%%%%%%%%%
%%%%%%%%%%%%%%%%%%%%%%%%%%%%%%%%%%%%
%%%%%%%%%%%%%%%%%%%%%%%%%%%%%%%%%%%%
%%%%%%%%%%%%%%%%%%%%%%%%%%%%%%%%%%%%
%%%%%%%%%%%%%%%%%%%%%%%%%%%%%%%%%%%%
\section{Strict upper bound, 
\texorpdfstring{$\pc<\pco$}{pc<pco}}
\label{sec:main}

As outlined in Section \ref{sec:outline}, the proof of \eqref{T_pcp} relies on a certain model of {\it enhanced}  oriented site percolation on $\mathbb Z^{2}$, which, roughly speaking, is the usual oriented site percolation model, but with the possibility of opening some vertical edges of length two. The interesting feature (and difficulty) of this model is that these additional edges are strongly correlated. In fact, in each row, we will open {\it all} such edges with some positive probability
(or else they are all closed), independently of other rows. Our main result is that, no matter how small this probability is, this strictly decreases the critical parameter for the existence of an infinite, open path starting from the origin.

%%%%%%%%%%%%%%%%%%%%%%%%%%%%%%%%%%%%
%%%%%%%%%%%%%%%%%%%%%%%%%%%%%%%%%%%%
%%%%%%%%%%%%%%%%%%%%%%%%%%%%%%%%%%%%
%%%%%%%%%%%%%%%%%%%%%%%%%%%%%%%%%%%%
%%%%%%%%%%%%%%%%%%%%%%%%%%%%%%%%%%%%
%%%%%%%%%%%%%%%%%%%%%%%%%%%%%%%%%%%%
%%%%%%%%%%%%%%%%%%%%%%%%%%%%%%%%%%%%
\subsection{Enhanced oriented percolation} 
In this subsection, we perform the first step of Section \ref{sec:outline}. Namely, we define our auxiliary model of interest more precisely and state our main result concerning its behavior. Fix two parameters $p,q\in [0,1]$. All sites $(x,n)\in \mathbb Z^{{2}}$ are {\it open} with probability $p$,  independently of each other, and 
all oriented edges $((x,n),(x+1,n))$ and 
$((x,n),(x,n+1))$ 
of length one 
are {\it open} with probability 1. 
Additionally, independently for each $n \in\bbZ$, all the oriented edges $((x,2n),(x,2n+2))_{x\in\bbZ}$ 
of length two 
are \emph{open} 
(all at once)
with probability $q$. Edges and sites which are not open are  \emph{closed}.

A \emph{path} is a sequence of vertices $(x_i,n_i)_{i=0}^k$ such that $((x_i,n_i),(x_{i-1},n_{i-1}))$ is an edge for each $i\in\{1,\dots,k\}$ (regardless whether it is open or closed). The path $(x_i,n_i)_{i=0}^k$ is \emph{open} if all its edges $((x_i,n_i),(x_{i-1},n_{i-1}))_{i=1}^k$ are open and the sites $(x_i,n_i)_{i=1}^k$ are open (if $k=0$, the path is open by convention). In other words, a path is open if all its
edges and vertices are open, except possibly
the first vertex. (We allow this possibility 
for technical convenience, as then 
we can concatenate paths independently.)

A path is called \emph{simple} if it is open and if, whenever 
an edge of length two is used, say $((x,2n),(x,2n+2))$, 
the vertex $(x,2n+1)$ is closed. 
That is, length-two edges are only used if necessary. 
Note that, given any two vertices, if there exists an open path between them, there also exists a simple path between them, and so we can restrict our attention to simple paths. 
This will be useful, as 
two simple paths cannot cross without sharing at least one vertex.

We denote the law of this model by $\mathbb P_{p,q}$. Note that it can be seen as a probability measure on $\{0,1\}^{\mathbb Z^{2}}\times \{0,1\}^{\mathbb Z}$. 
We write $(x,n)\rightarrow (y,m)$ for the event that there exists an open path from $(x,n)$ to $(y,m)$. Likewise, $(x,n)\rightarrow \infty$ denotes the event that there exists an infinite open path starting from $(x,n)$. Also, given $A,B,C\subset \mathbb Z^{2}$, let 
$A\to[B]C$ denote the event that some site in $A$ is connected to some site in $C$ by an open path contained in $B$. {In this notation, we omit $B$ if it is equal to $\bbZ^2$.} Given any $q\in [0,1]$, we define the critical parameter of this model as:
$$\pc(q) = \inf\big\{ p : \mathbb P_{p,q}((0,0)\rightarrow \infty)>0\big\}.$$
Note that, by definition, $\pc(0) = \pco$ is the  critical parameter for the classical model of oriented site percolation. Our main result is the following. 
\begin{theorem} \label{theo:pq}
For any $q>0$, we have that 
$\pc(q) < \pc(0)=\pco$.
\end{theorem}
We will prove this result in the remaining subsections{, but let us first deduce \eqref{T_pcp} of Theorem \ref{th:main} from Theorem \ref{theo:pq}.}

\begin{proof}[Proof of \eqref{T_pcp}]
{By Theorem \ref{theo:pq}, we can fix $p<\pco$ such that $\pc(p')<p$, with $p'=1-\sqrt{1-p}$. We couple Catalan percolation with parameter $p$ and our enhanced oriented percolation model with parameters $(p,p')$ as follows, similarly to Section \ref{sec:coupling:OP} (see Figure \ref{fig:coupling 2 edge}). Fix $n\ge 3$. For Catalan percolation, we declare the edges $\{i,j\}$ for $j\ge i+3$ open independently with probability $p$. For enhanced oriented percolation, we declare site $(i,j)$ for $i\ge 0$ and $j\in[0,n-3-i]$ open if and only if the Catalan edge $\{j,n-i\}$ is open. We further consider independent Bernoulli random variables $\xi_j,\xi_j'$ with parameter $p'$ for $j\in\bbZ$. For $j\in\bbZ$, the length two Catalan edge $\{j,j+2\}$ is open if and only if $\xi_j+\xi_j'\neq 0$, which has probability $p=1-(1-p')^2$. For any $j\in\bbZ$, 
to incorporate the enhancement, 
we further declare the edge $((i,2j),(i,2j+2))$ open for all $i\in\bbZ$ if and only if $\xi_{2j}=1$.

It is not hard to check that, if $(0,0)\to(i,j)$ occurs with $i+j\in\{n-4,n-3\}$ and $\xi'_{j}=\xi'_{j+2}=1$, then $\{0,n\}$ is occupied in Catalan percolation.  Indeed, by induction, the Catalan edge corresponding to each site in the path from the origin to $(i,j)$ is occupied. Consider the event that the origin reaches $\ell^1$ distance at least $n-4$ in enhanced oriented percolation:
\[\cX=\bigcup_{i+j\in\{n-4,n-3\}}\{(0,0)\to (i,j)\}.\]
By the above considerations, and independence, we have the uniform bound
\[\bfP_p(\{0,n\}\text{ is occupied})\ge \bbP_{p,p'}(\cX)(p')^2\ge \bbP_{p,p'}((0,0)\to\infty)(p')^2>0.\]
Recalling \eqref{eq:def:pc}, this yields \eqref{T_pcp}, as desired.}
\end{proof}

\begin{figure}[h]
\centering
\includegraphics[width = 0.9\textwidth]{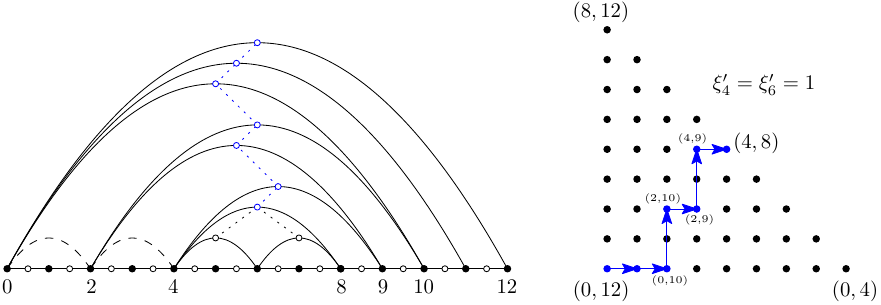}
\caption{An example of the coupling, with $n=12$.
{\it On the left:} A sequence of occupied edges. Each edge of length greater 
than $4$ is obtained by extending the edge underneath either by one 
in either direction (using an initially occupied, nearest-neighbour
edge of the form $\{i,i+1\}$), or by two to the left if its left endpoint is even (using a 
length-two
Catalan edge of the form $\{2i,2i+2\}$). 
The occupied edges are drawn as solid lines. 
The Catalan edges $\{0,2\}$ and $\{2,4\}$
are drawn as dashed lines. 
Note that $\{2,4\}$ allows $\{2,9\}$ to become
occupied after $\{4,9\}$ becomes occupied. 
Similarly, $\{0,2\}$ allows $\{0,10\}$ to become
occupied after $\{2,10\}$ becomes occupied. 
{\it On the right:} The coupled path in the oriented site percolation model, 
along with the relevant values of $\xi_j'$. The blue oriented path on the right is 
a rotation of the blue dotted path on the left by 135 degrees.
Each vertical, length-two edge along this path corresponds to the use
of a length-two Catalan edge.
The final steps from an edge of length four to two edges of length 
two are shown in black.}
\label{fig:coupling 2 edge}
\end{figure}

%%%%%%%%%%%%%%%%%%%%%%%%%%%%%%%%%%%%
%%%%%%%%%%%%%%%%%%%%%%%%%%%%%%%%%%%%
%%%%%%%%%%%%%%%%%%%%%%%%%%%%%%%%%%%%
%%%%%%%%%%%%%%%%%%%%%%%%%%%%%%%%%%%%
%%%%%%%%%%%%%%%%%%%%%%%%%%%%%%%%%%%%
%%%%%%%%%%%%%%%%%%%%%%%%%%%%%%%%%%%%
%%%%%%%%%%%%%%%%%%%%%%%%%%%%%%%%%%%%
\subsection{Edge speeds}
The {second} step in the proof of Theorem \ref{theo:pq} (see Section \ref{sec:outline}) is to show that if $p=\pco$ and $q>0$, then the open cluster of the origin spreads out at positive speed as the time (i.e.\ vertical) coordinate increases. This result is mostly classical, but we include its proof in our setting in Appendix \ref{app:classics} for the reader's convenience.

We start with some notation. Fix $p\in(0,1)$ and $q\in [0,1)$. For~$A \subset \mathbb Z$ and~$m,n \in \bbZ$ with~$m \le n$, define
\begin{equation}\label{eq_def_xi}
	\mathcal \xi_{m,n}(A) := \left\{x\in\bbZ:A\times\{m\}\to \{(x,n)\}\right\}.
\end{equation}
In words, $\xi_{m,n}(A)$ is the set of $x$-coordinates of sites
at level $n$ that are accessible from sites
at level $m$, whose $x$-coordinates are in $A$. 

For $n\ge 0$, we also write
\begin{align*}
\xi_{n}(A)&{} := \xi_{0,n}(A),&	r_n &{}:= \max \xi_n(-\bbN),&l_n &{}:= \min \xi_n(\bbN).
\end{align*}
The following is a consequence of Liggett's subadditive theorem (see Appendix \ref{app:edge}).
\begin{lemma}[Existence of edge speeds]
\label{lem_alpha_beta}
	If~$p \in (0,1)$ and~$q \in [0,1)$, there exist~$\alpha(p,q) \in [-\infty,\infty)$ and $\beta(p,q) \in (0,\infty]$ such that almost surely under $\bbP_{p,q}$,
	\begin{align*}
		\frac{r_{{2}n}}{{2}n} &\xrightarrow{n \to \infty} \alpha(p,q),&\frac{l_{{2}n}}{{2}n} &\xrightarrow{n \to \infty} \beta(p,q).
	\end{align*}
\end{lemma}
{The \emph{edge speeds} $\alpha$ and $\beta$ from Lemma \ref{lem_alpha_beta} satisfy the following strict inequalities proved in Appendix \ref{app:edge}.}
\begin{lemma}[Strict inequalities for edge speeds]
\label{lem:alpha:strict}
If~$q > 0$, then
\begin{align*}
\alpha(\pc(0),q) &{}> 1,&\beta(\pc(0),q) &{}< 1.
\end{align*}
\end{lemma}

%%%%%%%%%%%%%%%%%%%%%%%%%%%%%%%%%%%%
%%%%%%%%%%%%%%%%%%%%%%%%%%%%%%%%%%%%
%%%%%%%%%%%%%%%%%%%%%%%%%%%%%%%%%%%%
%%%%%%%%%%%%%%%%%%%%%%%%%%%%%%%%%%%%
%%%%%%%%%%%%%%%%%%%%%%%%%%%%%%%%%%%%
%%%%%%%%%%%%%%%%%%%%%%%%%%%%%%%%%%%%
%%%%%%%%%%%%%%%%%%%%%%%%%%%%%%%%%%%%
\subsection{Crossing boxes in the supercritical regime}
\label{subsec:crossing}
The third step in the proof of Theorem \ref{theo:pq} (see Section \ref{sec:outline}) is to establish that certain boxes are likely to be crossed. For this we need some geometric notation.

Given two vectors $u,v\in \bbR^2$ with $\det(u,v)>0$, we denote by \[R(u,v)=([0,1)u+[0,1)v)\cap\bbZ^2\]
the parallelogram generated by $u,v$. For 
such a parallelogram $R=R(u,v)$, we define
\begin{align*}
\cC_\rightarrow(R)&=\{[0,1)v\to[R]u+[0,1)v\},&
\cC_\uparrow(R)&=\{[0,1)u\to[R]v+[0,1)u\},\\
\cC_\leftarrow(R)&=\{u+[0,1)v\to[R]{}[0,1)v\},
\end{align*}
that is, the events that $R$ is crossed in each of the three directions by an open path. Note that here we use the convention that the start and end points of the crossing paths are allowed to be at Euclidean distance smaller than one from the boundary of $R$, as long as they are inside $R$. Also in the whole remainder of this section, we use the convention that any inequality of the form $\mathbb P_{p,q}(\cC_\uparrow(R))> \theta$, should be interpreted as the fact that the probability to cross {\it any translate} of $R$ in the upward direction is larger than $\theta$ (and similarly for crossings in the other directions $\rightarrow$ and $\leftarrow$). All proofs will generally be done only for one instance of the parallelograms, and it should be clear that, with minor modification in each case, they extend to any translate. The next statement is proved in Appendix \ref{app:box} by classical means from \cite{Durrett84}.
\begin{lemma}[Annealed box crossing]
\label{lem:crossing}
Let~$p \in (0,1)$ and~$q \in [0,1)$ be such that~$0< \beta(p,q) \le \alpha(p,q) < \infty$. Then, for any~$\delta > 0$ and~$\varepsilon > 0$, the following holds for~$n$ large enough. Letting
\begin{equation}
\label{eq:def:u:v}u = (\delta n,0),\quad v = (\alpha(p,q)\cdot n,n),\quad w = (\beta(p,q) \cdot n,n),
\end{equation}
we have
\begin{align*}
\bbP_{p,q}(\cC_\uparrow(R(u,v))) &{}> 1-\varepsilon,& \bbP_{p,q}(\cC_\uparrow(R(u,w))) &{}> 1-\varepsilon.
\end{align*}
\end{lemma}

%%%%%%%%%%%%%%%%%%%%%%%%%%%%%%%%%%%%
%%%%%%%%%%%%%%%%%%%%%%%%%%%%%%%%%%%%
%%%%%%%%%%%%%%%%%%%%%%%%%%%%%%%%%%%%
%%%%%%%%%%%%%%%%%%%%%%%%%%%%%%%%%%%%
%%%%%%%%%%%%%%%%%%%%%%%%%%%%%%%%%%%%
%%%%%%%%%%%%%%%%%%%%%%%%%%%%%%%%%%%%
%%%%%%%%%%%%%%%%%%%%%%%%%%%%%%%%%%%%
\subsection{Crossing bad times: Russo--Seymour--Welsh theory}\label{sec:RSW}

The fourth step in the proof of Theorem \ref{theo:pq} (see Section \ref{sec:outline}) deals with bad times, that is, time intervals when insufficiently many length-two edges are open. Since the length-two edges fail to provide enough help, we 
will completely disregard them. 
As such, this brings us to crossing estimates for the classical oriented site percolation model. These are based on the following result,  which summarises the main content of \cite{Duminil-Copin17}.
\begin{theorem}\cite[Theorem~1.3, Proposition~4.2, Remark~4.4]{Duminil-Copin17}.
\label{th:DCTT}
There exists $\varepsilon>0$ such that, for any $m\in\bbN$ large enough, there exists $w_m\in[\varepsilon m^{2/5},m^{1-\varepsilon}]\cap\bbZ$ such that
\begin{align*}\bbP_{\pco,0}(\cC_\rightarrow(R(3u,v))&\ge \varepsilon,&
\bbP_{\pco,0}(\cC_\uparrow(R(u,3v)))&\ge\varepsilon,&\bbP_{\pco,0}(\cC_\leftarrow(R(3u,v)))&\ge \varepsilon
\end{align*}
with $u=(w_m,-w_m)$ and $v=(m,m)$.
\end{theorem}

Next, we will adapt the geometry of the crossings provided by Theorem \ref{th:DCTT} to suit our needs.

\begin{corollary}
\label{cor:renorm:bad}
There exists $\varepsilon>0$, such that for any $m\in \bbN$ large enough, there exists an integer $\ell \in [\varepsilon m^{2/5},m^{1-\varepsilon}]$, for which 
\begin{align*}
\bbP_{\pco,0}\left(\cC_\uparrow(R((\ell,0),(m-4 \ell ,m)))\right)&{}\ge \varepsilon,&\bbP_{\pco,0}\left(\cC_\uparrow(R((\ell,0),(m+4 \ell ,m)))\right)&{}\ge \varepsilon.
\end{align*}
\end{corollary}
\begin{proof}
Let $M=\lceil  \tfrac{m}{20}\rceil$. Recalling Theorem \ref{th:DCTT}, set $u=(w_M,-w_M)$, and $v=(M,M)$. For $i\in\bbZ$ let $R_i=i(2v-2u)+R(u,3v)$ and $S_i=i(2v-2u)+R(3u,v)$. Consider the event 
\[\cA=\bigcap_{i=0}^{9}\left(\cC_\uparrow(R_i)\cap\cC_\leftarrow(S_i)\right),\]
and note, as illustrated in Figure \ref{fig:RSW}, 
that
\[\mathcal A\subset \cC_\uparrow(R((L,0),20(v-u))),\]
with $L = 3M+w_M -\theta (3M-w_M)$, and $\theta = \tfrac{2M-2w_M}{2M+2w_M}$. 

By the Harris inequality \cite{Harris60} and Theorem \ref{th:DCTT}, we have $\bbP_{\pco,0}(\cA)\ge \varepsilon$, for some fixed $\varepsilon>0$, and any $m$ large enough. 
Moreover, 
$$L \le 3M+w_M - (3M-w_M)\left(1-2\frac{w_M}{M}\right) 
\le 8w_M, $$
and 
$$20(v-u) = 20 (M+w_M,M+w_M) - 40 (w_M,0).$$ 
In particular, letting $\ell = 10w_M$, one has for $m$ large enough, that 
$$
\bbP_{\pco,0}(\cC_\uparrow( 
R((\ell,0),(m-4\ell,m))))\ge \bbP_{\pco,0}(\cC_\uparrow(R((L,0),20(v-u))))\ge\bbP_{\pco,0}(\cA) \ge \varepsilon.
$$ 
Similarly,
$$\mathbb P_{\pc^0,0}\big( \mathcal C_\uparrow(R((\ell,0),(m+4\ell,m)))\big) \ge \varepsilon,$$
which completes the proof.  
\end{proof}

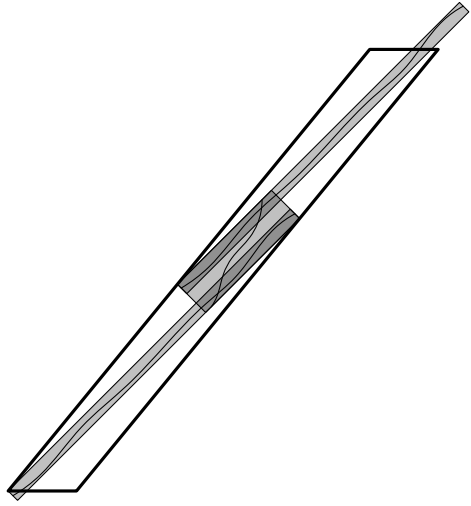
\begin{figure}
\centering
\begin{tikzpicture}[line cap=round,line join=round,>=triangle 45,x=0.125cm,y=0.125cm]
\draw[fill=black,fill opacity=0.25] (0,0) -- (30,30) -- (29,31) -- (-1,1) -- cycle;
\draw[fill=black,fill opacity=0.25] (20,20) -- (30,30) -- (27,33) -- (17,23) -- cycle;
\draw[fill=black, fill opacity=0.25] (18,22) -- (48,52) -- (47,53) -- (17,23) -- cycle;
\draw [very thick] (6.27,1)-- (44.73,48)-- (37.45,48)-- (-1,1)--cycle;
\draw (-0.5,0.5) to [curve through={(2.4,2.7)..(6.8,7.8)..(11.6,12.2)..(20,20.9)..(23.2,23.8)..(27.4,29)}] (29.5,30.5);
\draw (20.6,20.6) to [curve through={(21.2,21.7).. (21.9,24.1)..(22.8,25.8)..(24.7,28)..(25.5,29.4)}] (26,32);
\begin{scope}[shift={(17.9,22.1)}]
\draw (-0.5,0.5) to [curve through={(2.4,2.7)..(6.8,7.8)..(11.6,12.2)..(20,20.9)..(23.2,23.8)..(27.4,29)}] (29.5,30.5);
\end{scope}
\end{tikzpicture}
\caption{If the three shaded rectangles of dimensions either 
$3w_M\times M$ or $w_M\times 3M$ are crossed in the appropriate directions, 
then the thickened parallelogram is also crossed.}
\label{fig:RSW}
\end{figure}

%%%%%%%%%%%%%%%%%%%%%%%%%%%%%%%%%%%%
%%%%%%%%%%%%%%%%%%%%%%%%%%%%%%%%%%%%
%%%%%%%%%%%%%%%%%%%%%%%%%%%%%%%%%%%%
%%%%%%%%%%%%%%%%%%%%%%%%%%%%%%%%%%%%
%%%%%%%%%%%%%%%%%%%%%%%%%%%%%%%%%%%%
%%%%%%%%%%%%%%%%%%%%%%%%%%%%%%%%%%%%
%%%%%%%%%%%%%%%%%%%%%%%%%%%%%%%%%%%%
\subsection{Oriented percolation in a random environment}

Finally, we are ready to proceed to the final step in the proof of Theorem \ref{theo:pq} (recall Section \ref{sec:outline}).

\begin{proposition}[Renormalisation]\label{prop.final}
Let $\varepsilon>0$ and $0<\beta<1<\alpha$ be given. Let $\rho=\rho(\alpha,\beta)>0$, be such that 
$\beta+ 2\rho <1$, $\alpha -2\rho >1$, and $\alpha - \beta \ge 12\rho$. 
Then there exist $\varepsilon'>0$, such that for any $m\ge 1$, any  $\ell \in [1,  \tfrac{\rho m}{2}]$, and any $p,q\in [0,1)$ the following holds. If 
\begin{align}\label{hyp1}
 \bbP_{p,q}\left(\cC_\uparrow\left(R((m\rho,0),(m\alpha,m)
 )\right)\cap \cC_\uparrow\left(R((m\rho,0),(m\beta,m)
 )\right)\right)&{}>
 1-\varepsilon',\\
\intertext{and}
\label{hyp2}
\bbP_{p,0}\left(\cC_\uparrow\left(R((\ell,0),(m+4 \ell,m))\right)\cap  \cC_\uparrow\left(R((\ell,0),(m-4 \ell,m))\right)\right)&{}> 
\varepsilon,
\end{align}
then $\pc(q) < p$. 
\end{proposition}

Before proving Proposition \ref{prop.final}, let us conclude the proof of the main result of this section, Theorem \ref{theo:pq}.
\begin{proof}[Proof of Theorem \ref{theo:pq}]
Fix $p=\pco$ and $q>0$. By Lemmas \ref{lem:alpha:strict} and \ref{lem_alpha_beta}, we have $0<\beta<1<\alpha<\infty$, setting $\alpha=\alpha(p,q)$ and $\beta=\beta(p,q)$. Fix $\varepsilon'$ provided by Proposition \ref{prop.final} for $\varepsilon$ given by $(\tilde \varepsilon)^2$, where $\tilde \varepsilon$ is the value of $\varepsilon$ provided by Corollary \ref{cor:renorm:bad}. It then suffices to find $m\ge 1$ and $\ell\in[1,\tfrac{\rho m}{2}]$ so that \eqref{hyp1} and \eqref{hyp2} hold. By Lemma \ref{lem:crossing} and a union bound, \eqref{hyp1} is satisfied for any $m$ large enough. Finally, by Corollary \ref{cor:renorm:bad} and the Harris inequality \cite{Harris60}, for any $m$ large enough we can choose $\ell\in[1,\tfrac{\rho m}{2}]$ so that \eqref{hyp2} holds.
\end{proof}
The proof of Proposition \ref{prop.final} relies on the 
recent result \cite[Theorem 8.2]{Hilario23}.

\begin{theorem}[Oriented percolation with geometric defects, \cite{Hilario23}]
\label{th:HSST}
Let $p,\delta\in(0,1)$ and $\xi=(\xi_i)_{i\in\bbN}$ be a sequence of independent random variables with $\bbP(\xi=k)=(1-\delta)\delta^k$ for $k\in\bbN$. Endow $\bbN^2$ with the oriented edge set $E=\{((i,n),(i,n+1)),((i,n),(i+1,n+1)):(i,n)\in\bbN^2\}$. Conditionally on the \emph{environment} $\xi$, we declare each edge from $(i,n)$ to be open independently with probability $p^{\xi_n+1}$ for all $(i,n)\in\bbN^2$. Denoting the law of this process by $\bbP_p^\xi$, the following holds. There exists $\varepsilon>0$ such that if $\delta\le \varepsilon$ and $p\ge 1-\varepsilon$, then for almost every environment $\xi$, under $\bbP_p^\xi$, there is an infinite open path starting at the origin with positive probability.
\end{theorem}
\begin{proof}[Proof of Proposition \ref{prop.final}] 
Let $0<\beta<1<\alpha$, $\rho\in(0,\min((\alpha-\beta)/12,(\alpha-1)/2,(1-\beta)/2))$, $1/2\ge\varepsilon,\varepsilon'>0$, $m\ge 1$, $\ell \in [1,\tfrac{\rho m}{2}]$ and $p,q$ be given, so that~\eqref{hyp1} and~\eqref{hyp2} are satisfied. The value of $\varepsilon'$ is assumed small enough so that some conditions imposed later in the proof hold. Let 
\begin{align}
\label{eq:def:RaRb}
R^\alpha&{}=R((m\rho,0),(m\alpha,m)),&R^\beta&{}=R((m\rho,0),(m\beta,m)).
\end{align}
First note that it suffices to show that the probability that the origin is connected to infinity by an open path is positive under $\mathbb P_{p,q}$, since then by continuity of the probabilities in~\eqref{hyp1} and~\eqref{hyp2} as functions of $p$, this would remain true for a smaller value of $p$. 

The strategy is to compare our model with the model of oriented bond percolation in random environment 
considered in Theorem \ref{th:HSST}. 
Here the role of the random environment is played by the state of all length-two edges, whose associated sigma-field 
is denoted by $\mathcal E$. Declare an integer $n\ge 0$ \emph{good} if 
\begin{equation}\label{good.integer}
\bbP_{p,q}\Big(\cC_\uparrow((0, \frac{nm}2)+ R^\alpha)\cap\cC_\uparrow((0,\frac{nm}2) + R^\beta)\ \Big| \ \mathcal E \Big)\ge 1-\sqrt{\varepsilon'},
\end{equation}
and call it \emph{bad} otherwise. Denoting by $\mathbb Q_q$ the law of all length-two edges and using \eqref{hyp1}, one has 
\begin{align*}
\varepsilon'  \ge 1-\mathbb P_{p,q}\Big(\cC_\uparrow((0, \frac{nm}2)+ R^\alpha)\cap\cC_\uparrow((0,\frac{nm}2) + R^\beta)\Big) \ge \mathbb Q_q(n \text{ is bad})\times \sqrt{\varepsilon'},  
\end{align*} 
from which we infer that for any $n\in\bbN$,  
\[\mathbb Q_q(n \text{ is good})\ge 1 -  \sqrt{\varepsilon'}.\]
It follows that the random variables $(\1\{n\text{ is good}\})_{n\ge 0}$, form a sequence of $1$-dependent identically distributed Bernoulli random variables, with mean larger than $1- \sqrt{\varepsilon'}$. Thus, by the Liggett--Schonmann--Stacey theorem~\cite[Theorem~0.0]{Liggett97}, one can ensure the existence of independent Bernoulli random variables $(X_n)_{n\ge 0}$, with mean $1-\delta$, such that for all $n\in\bbN$,
\[\1\{n\text{ is good}\} \ge X_n,\]
where $\delta>0$ can be taken arbitrarily close to $0$, by choosing $\varepsilon'$ small enough. We also set $X_{-1} = 1$. We further assume $(X_n)_{n\ge 0}$ to be constructed on the probability space of $\bbP_{p,q}$ in such a way that they are independent of the sigma-algebra generated by the set of open sites of $\bbZ^2$.

Next, we identify the intervals of good times, by defining the sequence $(\tau_n:n\ge -1)$ inductively, by $\tau_{-1} = -1$, and  
\begin{align*}
\tau_{n} &{}= \inf \{k\ge \tau_{n-1} + 1 : X_k = 1\},
&\xi_n  &{}= \tau_n - \tau_{n-1} - 1,
\end{align*}
for $n\in\bbN$. By construction the $(\xi_n)_{n\ge 0}$ are independent random variables with common law given by  
\begin{equation}\label{law.xin}
P(\xi_n = k) = \delta^k (1-\delta), \quad \text{for all }k\ge 0.
\end{equation}

Now we define a renormalized lattice, similarly to \cite[Section 9]{Durrett84}, at least on good rows (corresponding to integers $n$ such that $X_n = 1$), and using also a notion of {\it stretched bonds}, to accommodate the crossing of consecutive bad rows.

We define inductively the new vertices $(z_{i,n})_{i\ge 0, n\ge 0 }$ (in $[0,\infty)^2$) of our renormalized lattice as follows (see Figure \ref{crossing2}).
First 
$$z_{i,0} = im\cdot \big(\frac{\alpha-\beta}{2} - 2\rho,0\big),\quad \text{for }i\ge 0,$$
and note that by definition of $\rho$, one has $\tfrac{\alpha-\beta}{2} - 2\rho\ge 4 \rho$. 
Next, given $n\ge 0$, we start by defining for $i\ge 0$, 
$$\widetilde z_{i,n} = z_{i,n} + \frac {m\xi_n} 2 \cdot (1,1),$$  
and then let
\begin{equation}\label{zn+1i}
z_{i,n+1} =  \widetilde z_{i,n} + \frac m2 \cdot (2\rho +  \beta,1 ).  
\end{equation}
Now we consider a new lattice $\mathbb N^2$, with edge set $E$ from Theorem \ref{th:HSST}. 
For any $n\ge 0$ and $i\ge 0$, we declare the vertex $(i,n)$ \textit{open} if either $\xi_n = 0$ (in which case $z_{i,n} = \widetilde z_{i,n}$), or, when $\xi_n\ge 1$, if the following two events hold without using any length-two edge (see Figure \ref{crossing2}):  
\begin{equation}\label{open.v.1}
\bigcap_{j=0}^{\xi_n-1}  \cC_\uparrow\big(z_{i,n,j}  + R((\ell,0),(m+4(-1)^j \ell,m)) \big),
\end{equation} 
and 
\begin{equation}\label{open.v.2}
\bigcap_{j=0}^{\xi_n-1}  \cC_\uparrow\big(z'_{i,n,j}  + R((\ell,0),(m+4(-1)^{j+1} \ell,m)) \big),
\end{equation} 
where for all $j\ge 0$, 
\begin{align*}
z_{i,n,j} &{}= z_{i,n} + (\rho m - \ell,0) + j\left(\frac m2, \frac m2\right) +3\ell \cdot  \1\{j \text{ is odd}\} \cdot (1,0),\\
\intertext{and}
z'_{i,n,j} &{}= z_{i,n} + (2\rho m,0) + j\left(\frac m2, \frac m2\right) -3\ell \cdot  \1\{j \text{ is odd}\} \cdot (1,0).
\end{align*}

\begin{figure}
\centering
\includegraphics[width = 0.8\textwidth]{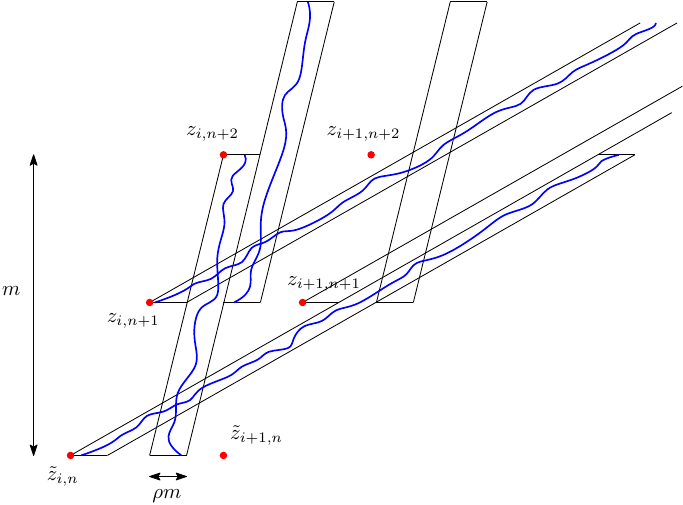}
\caption{Illustration of how the crossings of different parallelograms 
may be glued together, when $\xi_{n+1}= 0$. In this example the two 
edges emanating from both $(i,n)$ and $(i,n+1)$ are open 
since the corresponding parallelograms are crossed vertically (by blue paths).}
\label{crossing}
\end{figure}

\begin{figure}
\centering
\includegraphics[width = 0.9\textwidth]{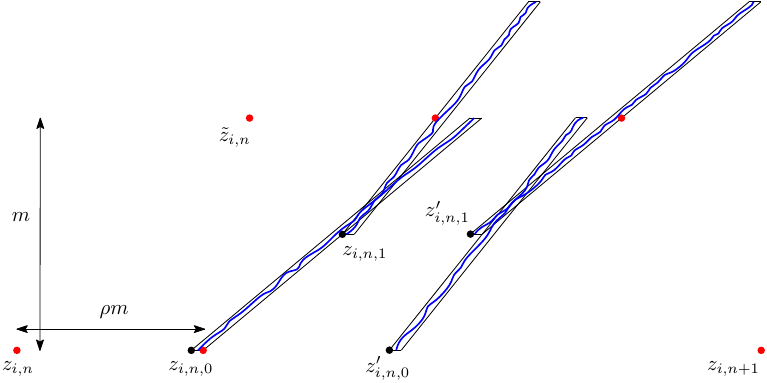}
\caption{Illustration of the definitions in~\eqref{open.v.1} and~\eqref{open.v.2}. 
In this example $\xi_n = 2$, and $(i,n)$ is open, since the 
corresponding parallelograms are crossed vertically (by blue paths). 
It is also apparent that the last parallelograms pass between 
$\widetilde z_{i,n} + (\rho m,0)$ and $\widetilde z_{i,n} +2(\rho m,0)$, which are marked by red dots.}
\label{crossing2}
\end{figure}

\noindent Furthermore, we say that the edge $((i,n),(i+1,n+1))$ is \textit{open}, if the following event holds   
\begin{equation}
\label{eq:def:open:right}\cC_\uparrow(\widetilde z_{i,n}+ R^\alpha)\ \cap \ \cC_\uparrow(\widetilde z_{i,n} +(2\rho m,0) +  R((m\rho,0),(m\beta/2,m/2)
)),\end{equation}
and similarly we say that the edge $((i,n),(i,n+1))$ is open if the following event holds 
\begin{equation}
\label{eq:def:open:left}\cC_\uparrow(\widetilde z_{i,n}+ R((m\rho,0),(m\alpha/2,m/2)
)\ \cap \ \cC_\uparrow(\widetilde z_{i,n} +(2\rho m,0) +  R^\beta).
\end{equation}
An \textit{open path} in the new lattice is a sequence $(i_1,n_1),\dots,(i_k,n_k)$ (possibly with $k=\infty$), such that for each $1\le j<k$, 
$(i_j,n_j)$ is open and $((i_j,n_j),(i_{j+1},n_{j+1}))$ is an open edge.

The proof of Proposition \ref{prop.final} is complete if we prove the following two lemmas.\end{proof}

\begin{lemma}
\label{lem:renormalisation:proba}
    For almost every realization of the environment and the $(X_n)_{n\ge 0}$ variables, whose sigma-algebra is denoted $\cF^X$, the following holds. Under $\bbP_{p,q}(\cdot|\cE,\cF^X)$, the origin is in an infinite open path in the renormalized lattice with positive probability.
\end{lemma}
\begin{lemma}
\label{lem:renormalisation:deterministic}
    If there exists an infinite open path in the renormalized lattice, then there is an infinite open path in the original lattice.
\end{lemma}

\begin{proof}[Proof of Lemma \ref{lem:renormalisation:proba}]
For $i,n\ge 0$, define the random variables
\begin{equation}
\label{eq:Yni}Y_{i,n}= \1\{(i,n) \text{ is open}\}.
\end{equation}
For an edge $e=(a,b)\in E$, we define
\begin{equation}\label{def:Ze}
Z_e =
\begin{cases}
\1\{e \text{ is open}\} & \text{if }Y_a=Y_b = 1\\
1 & \text{otherwise.} 
\end{cases}
\end{equation}
Let $\mathcal F^Y$ be the sigma-algebra generated by $(Y_{i,n})_{(i,n)\in\bbN^2}$. Set $\mathbf P=\bbP_{p,q}(\cdot|\cE,\cF^X,\cF^Y)$ and note that $\bbP_{p,q}(\cdot\mid\cE,\cF^X)$ identifies with the product Bernoulli measure with parameter $p$ on $\{0,1\}^{\bbZ^2}$. 
Note that by the Harris inequality and the definition of a good integer (recall~\eqref{good.integer}), for any edge $e=(a,b)\in E$, one almost surely has
\begin{align*}
\mathbf P( Z_e =1)  &{}= \1\{Y_aY_b=0\}+\1\{Y_aY_b=1\}\cdot \bbP_{p,q}(Z_e=1
\mid \cE,\cF^X,Y_a=Y_b=1
)\\
&\ge \1\{Y_aY_b=0\}+\1\{Y_aY_b=1\}\cdot \bbP_{p,q}(Z_e=1
\mid \cE,\cF^X)\ge 1- \sqrt{\varepsilon'}.
\end{align*}
Moreover, we claim that, almost surely under $\mathbf P$, the random variables $(Z_e)_{e\in E}$ are 1-dependent. In fact, 
under $\bbP_{p,q}(\cdot|\cE,\cF^X)$, the variables $(Y_{i,n})_{(i,n)\in\bbN^2}$ and $(Z_e)_{e\in E}$ are jointly 1-dependent. Indeed, for $V_1,V_2\subset \bbN^2$ and $E_1,E_2\subset E$ with $E_1,V_1$ not incident with $E_2,V_2$, the vectors $(Y_v,Z_e)_{v\in V_1,e\in E_1}$ and $(Y_v,Z_e)_{v\in V_2,e\in E_2}$ are independent, because they depend on open sites 
(we have conditioned on $\cE$) in deterministic disjoint regions in space. In particular, $(Y_v)_{v\in\bbN^2}$ are independent under $\bbP_{p,q}(\cdot|\cE,\cF^X)$.

By 1-dependence of $(Z_e)_{e\in E}$ under $\mathbf P$ and \cite[Theorem~0.0]{Liggett97}, the following holds, for arbitrarily small $\delta'>0$, provided $\varepsilon'>0$ is small enough. We can find independent Bernoulli random variables $(\overline Z_e)_{e\in E}$ with parameter $1-\delta'$, such that \begin{equation}
\label{eq:ZZbar}Z_e\ge \overline Z_e,
\end{equation} 
for all $e\in E$. Since their law does not depend on $\cE,\cF^X,\cF^Y$, they are independent of these sigma-fields. 

As we already established, $(Y_{i,n})_{i,n\in\bbN}$ are independent and, 
by \eqref{hyp2} and the Harris inequality, $\bbP_{p,q}(Y_{i,n}=1\mid\cE,\cF^X)\ge \varepsilon^{2\xi_n}$.
Consequently one can define a sequence of independent Bernoulli random variables $(W_e)_{e\in E}$, so 
that for each $i,n\ge 0$, if $e$ and $f$ are the two edges emanating from $(i,n)$, then $W_e$ and $W_f$ have mean $\varepsilon^{4\xi_n}$, and satisfy  
\begin{equation}
\label{eq:YW}Y_{i,n}  \ge \max (W_e,W_f).
\end{equation}
Indeed, $x\ge 1-(1-x^2)^2$ for all $x\in[0,1/2]\cup\{1\}$ and $\varepsilon^2\le 1/2$.

We now declare an edge $e\in E$ to be \textit{good} 
if $W_e= \overline Z_e= 1$, which by independence between $W_e$ and $\overline Z_e$ holds with probability
\begin{equation}\label{probability}
\varepsilon^{4\xi_n}(1-\delta'),
\end{equation} 
independently for each edge $e$. Forgetting about the states of vertices, we end up with a new model of oriented bond percolation, 
which almost fits the setting of Theorem \ref{th:HSST}.
More precisely, we would be able to apply Theorem \ref{th:HSST}, if the factor $\varepsilon^4$ in~\eqref{probability} were replaced by $1-\delta'$. However, 
one can easily  recover the exact setting of~Theorem \ref{th:HSST} as follows. Fix $M$ such that $\varepsilon^{4/M}\ge 1- \delta'$. Then simply observe, recalling~\eqref{law.xin}, that $M\xi_n$ is stochastically dominated by a geometric random variable with mean that can be chosen arbitrarily close to zero by taking smaller $\varepsilon'$ if necessary (while still fixing $\delta'$ and $M$).

Hence, by Theorem \ref{th:HSST}, with positive probability, under $\bbP(\cdot|\cE,\cF^X)$, there is an infinite oriented path of good edges. Putting \eqref{eq:Yni}, \eqref{def:Ze}, \eqref{eq:ZZbar} and \eqref{eq:YW} together, we obtain that any such path yields an open path, concluding the proof of Lemma \ref{lem:renormalisation:proba}.
\end{proof}
\begin{proof}[Proof of Lemma \ref{lem:renormalisation:deterministic}]
It is useful to note first that, for any $i,n\ge 0$, one has 
\begin{equation}\label{zni}
z_{i,n} = z_{0,n} + im\cdot \big(\tfrac{\alpha-\beta}{2} - 2\rho,0\big),
\end{equation} 
which is immediate by induction on $n$. Also, recalling \eqref{eq:def:RaRb}, for each $(i,n)\in \mathbb N^2$, let 
$$R^1_{i,n}:=\widetilde z_{i,n} + R^\alpha,\qquad   
 R^2_{i,n}:=\widetilde z_{i,n} + 2(\rho m,0) + R^\beta,$$ and notice that these two parallelograms completely cross each other before reaching the level of $z_{i,n+1}$, in the sense that, at this level, the left-most point
of $R^1_{i,n}$ is on the right of the right-most point of $R^2_{i,n}$ (since $3\rho m + \tfrac{\beta m}2 \le \tfrac{\alpha m}2$, by the definition of $\rho$). 

Next,  assume that an edge emanating from $(i,n)\in\bbN^2$ is open, say $((i,n),(i,n+1))$.

{\bf Case 1 {\rm (assume $\xi_{n+1} = 0$)}.}
We need to verify that if any of the two edges emanating from $(i,n+1)$ is open,  then any vertical crossing of $R^2_{i,n}$ may be glued to the crossings of $R_{i,n+1}^1$ and $R_{i,n+1}^2$, before they reach the level of $z_{i,n+2}$. This can be checked using the following fact, see also Figure \ref{crossing}. 

Denoting by $z^1$ the horizontal coordinate of a point $z\in \mathbb R^2$, by \eqref{zn+1i}, one has 
\[
\widetilde z_{i,n}^1 + 2\rho m + \frac{\beta m}2 = z_{i,n+1}^1 + \rho m, 
\]
so that at the level of $z_{i,n+1}$, the parallelogram $R_{i,n}^2$ passes exactly in between $R_{i,n+1}^1$ and $R_{i,n+1}^2$, allowing all crossing paths to be glued together, see Figure \ref{crossing}. 

Moreover, the same reasoning applies if an edge emanating from $(i+1,n+1)$ is open, since
$$\widetilde z_{i,n}^1 + \frac{\alpha m}{2} = z_{i,n+1}^1 + \frac{(\alpha -\beta)m}{2} -\rho m =  z_{i+1,n+1}^1 + \rho m, $$ 
where the first equality follows from~\eqref{zn+1i}, and the second from \eqref{zni}. Thus, here again the parallelogram $R_{i,n}^1$ passes exactly between the $R_{i+1,n+1}^1$ and $R_{i+1,n+1}^2$, when arriving at the level of $z_{i+1,n+1}$ (see Figure \ref{crossing}).

{\bf Case 2: {\rm (assume  $\xi_{n+1}\ge 1$)}.}
First, we note that 
it may be seen that any vertical crossing of $R_{i,n}^1$ or 
$R_{i,n}^2$ can be glued to the crossings in~\eqref{open.v.1} or~\eqref{open.v.2}
(with $n+1$ instead of $n$). In the case of $R_{i,n}^2$, one can check that the first parallelogram in~\eqref{open.v.1}, 
\[r_1=z_{i,n+1}+(\rho m-\ell,0)+R((\ell,0),(m+4\ell,m)),\]
crosses $R_{i,n}^2$ before reaching the higher level, using that $\frac{\beta m}{2} +\rho m \le \frac m 2 +\ell$. Thus, at the level of $z_{i,n+2}$, the right-most point of $R_{i,n}^2$ is on the left of the left-most point of $r_1$. Moreover, the fact that all crossings in~\eqref{open.v.1} can be glued together is immediate by construction, see Figure \ref{crossing2}. 

Likewise, the fact that the first parallelogram in~\eqref{open.v.2} (with $i$ and $n$ replaced respectively by $i+1$ and $n+1$) intersects any vertical crossing of $R_{i,n}^1$ before reaching the higher level is guaranteed by the fact that $\frac{\alpha m}{2}\ge \rho m +\frac m 2 - \ell$, and thus, for the same reasons as before, 
all crossings in~\eqref{open.v.2} can be glued together.  

Therefore, 
it only remains to see that in case when a vertex, say $(i,n)$, is open and any of the two edges emanating from it is also open, the last crossings in~\eqref{open.v.1} and~\eqref{open.v.2} may be glued to the crossings of  $R_{i,n}^1$ or $R_{i,n}^2$. To see this, assume for concreteness that the edge $((i,n),(i,n+1))$ is open (the reasoning for $((i,n),(i+1,n+1))$ being analogous). 

Consider a crossing $\gamma_1$ of $R_{i,n}^2$ and a crossing $\gamma_2$ of the first half of $R_{i,n}^1$, whose existence is guaranteed by \eqref{eq:def:open:left}. 
Since $R_{i,n}^2$ and the first half of $R_{i,n}^1$ cross before reaching the level of $z_{i,n+1}$, the  
paths $\gamma_1$ and $\gamma_2$ also  
intersect before reaching this level. 
Thus, it can be seen, regardless of the parity of $\xi_n$, and since $\ell \le \rho m /2$, 
that the last parallelograms in~\eqref{open.v.1} and~\eqref{open.v.2} always pass between $\widetilde z_{i,n} + (\rho m,0)$ and 
$\widetilde z_{i,n} + (2\rho m,0)$ (see Figure \ref{crossing2}). In particular, when arriving at the level of $\widetilde z_{i,n}$, any crossing of these parallelograms, say $\gamma_3$, passes between the starting points of $\gamma_1$ and $\gamma_2$. Since $\gamma_1$ and $\gamma_2$ intersect before reaching the level of $z_{i,n+1}$, $\gamma_3$ has to intersect either $\gamma_1$ or $\gamma_2$ before they intersect for the first time, implying that open paths can be glued together.
\qedhere
\end{proof}

%%%%%%%%%%%%%%%%%%%%%%%%%%%%%%%%%%%%%%%
%%%%%%%%%%%%%%%%%%%%%%%%%%%%%%%%%%%%%%%
%%%%%%%%%%%%%%%%%%%%%%%%%%%%%%%%%%%%%%%
%%%%%%%%%%%%%%%%%%%%%%%%%%%%%%%%%%%%%%%
%%%%%%%%%%%%%%%%%%%%%%%%%%%%%%%%%%%%%%%
%%%%%%%%%%%%%%%%%%%%%%%%%%%%%%%%%%%%%%%
\appendix

%%%%%%%%%%%%%%%%%%%%%%%%%%%%%%%%%%%%
%%%%%%%%%%%%%%%%%%%%%%%%%%%%%%%%%%%%
%%%%%%%%%%%%%%%%%%%%%%%%%%%%%%%%%%%%
%%%%%%%%%%%%%%%%%%%%%%%%%%%%%%%%%%%%
%%%%%%%%%%%%%%%%%%%%%%%%%%%%%%%%%%%%
%%%%%%%%%%%%%%%%%%%%%%%%%%%%%%%%%%%%
%%%%%%%%%%%%%%%%%%%%%%%%%%%%%%%%%%%%
\section{Classical oriented percolation theory}
\label{app:classics}
\subsection{Edge speeds}
\label{app:edge}
\begin{proof}[Proof of Lemma \ref{lem_alpha_beta}]
	We claim that~$\bbE_{p,q}[r_n] < \infty$. To see this, let~$(r_n')_{n \ge 0}$ be defined by
	\begin{align*}&r_0':= \inf\{x \ge 0:\; (x+1,0) \text{ is closed}\}, \\
		&r_n':= \inf\{x \ge r_{n-1}':\; (x+1,n) \text{ is closed}\},\; n\ge 1.
	\end{align*}
	We clearly have~$r_n \le r_n'$ for all~$n$, and~$r_n' \sim \sum_{j=0}^n Y_j$, where~$Y_0, Y_1,\ldots$ are independent,~$\mathsf{Geometric}(1-p)$ random variables. The claim readily follows. We can now define
	\begin{equation}\label{eq_def_of_alpha}
	\alpha(p,q) := \inf_{n \ge 1} \frac{\mathbb E_{p,q}[r_{2n}]}{2n} \in [-\infty,\infty).
\end{equation}
The process $(r_{2n}-r'_0)_{n\in\bbN}$ has the properties required to apply Liggett's subadditive ergodic theorem~\cite[Theorem~VI.2.6]{Liggett05} to conclude that~$\frac{r_{2n}}{2n} \xrightarrow{n \to \infty} \alpha(p,q)$ almost surely.

The treatment of the second statement is similar, only simpler. Since~$l_n \ge 0$ {and equality is not almost sure}, we can directly define
\begin{equation*}
	\beta(p,q):= \sup_{n \ge 1} \frac{\mathbb E_{p,q}[l_{2n}]}{2n} \in [0,\infty].
\end{equation*}
The subadditive ergodic theorem then gives~$\frac{l_{2n}}{2n} \xrightarrow{n \to \infty} \beta(p,q){>0}$ almost surely.
\end{proof}

We next turn to proving Lemma \ref{lem:alpha:strict}, which requires some preparation.
\begin{lemma}
\label{lem_monotone_alpha}\leavevmode
		\begin{enumerate}
			\item If~$p$ satisfies~$\alpha(p,0) > - \infty$, then for any $q>0$,
				\begin{equation}
					\label{eq_alpha_mon}
				\alpha(p,q) - \alpha(p,0) \ge qp(1-p)^2.
				\end{equation}
			\item If~$p$ satisfies~$\beta(p,0) < \infty$, then for any $q>0$,
				\begin{equation}
					\label{eq_beta_mon}
					\beta(p,q) - \beta(p,0) \le qp(1-p)^2.
				\end{equation}
	\end{enumerate}
\end{lemma}
\begin{proof}
To prove~\eqref{eq_alpha_mon}, we fix any $p$ such that $\alpha(p,0)> -\infty$, and for $q\geq 0$ we let $r_n^q$ denote the random variable $r_n$ under $\bbP_{p,q}$. The proof will follow a similar strategy as that used to prove \cite[Equation (12)]{Durrett84}, proceeding in three main steps as follows.
\begin{enumerate}
\item\label{alpha:step:1} We first show that for any infinite sets~$A,B \subset \mathbb Z$ with~$A \subset B$ and~$\max B > \max A$, and for any~$m \le n$, we have
\begin{equation}\label{eqn:Durret rnA rnB inequality}
\bbE_{p,q}[\max\xi_{m,n}(B) - \max\xi_{m,n}(A)] \ge 1.
\end{equation}
\item\label{alpha:step:2} We then couple $\bbP_{p,q}$ with $\bbP_{p,q'}$ where $q'>q\geq 0$ under a common law $\bbP$ and use~\eqref{eqn:Durret rnA rnB inequality} to show that
		\begin{equation}
		\label{eq_for_step2}
	\bbE [r_{2n}^{q'} - r^{q}_{2n}] \geq 1-(1-(q'-q)p(1-p)^2)^n.
		\end{equation}
\item\label{alpha:step:3} We then tie this together to prove~\eqref{eq_alpha_mon}.
\end{enumerate}

	We start with Step \ref{alpha:step:1}. For concreteness, we take~$0 = m \le n$ (the proof is the same for~$0 < m \le n$). We proceed as in  \cite[Equation (13)]{Durrett84}. By the assumptions that~$A \subset B$ and~$x^* := \max B > \max A$, and by monotonicity of~$\xi_n(\cdot)$ with respect to set inclusion, we have
\[\max\xi_n(B) - \max\xi_n(A) \ge  \max\xi_n(B) - \max\xi_n(B \backslash \{x^*\}).\]
Using the definition of~$\xi_n(\cdot)$, we have
\begin{align*}
    \bbE_{p,q}[\max\xi_n(B) - \max\xi_n(B \backslash \{x^*\})] & = \bbE_{p,q}\left[\left(\max\xi_n(\{x^*\})- \max \xi_n(B \backslash \{x^*\}) \right)^+ \right].
\end{align*}
By  monotonicity and translation invariance, the right-hand side is larger than
\[\bbE_{p,q}\left[ \left( \max \xi_n(\{0\}) -  \max \xi_n(-\bbN\setminus \{0\}) \right)^+\right] = \bbE_{p,q}[\max \xi_n(-\bbN) - \max \xi_n(-\bbN\setminus \{0\}) ];\]
by translation invariance, the right-hand side equals 1. This proves~\eqref{eqn:Durret rnA rnB inequality}.

We now turn to Step \ref{alpha:step:2}. We couple $\bbP_{p,q}$ with $\bbP_{p,q'}$ under a common law $\bbP$ in the natural way: we first sample a site percolation configuration under $\bbP_{p,0}$, then for each set of vertical edges joining height $2n$ to $2n+2$ we independently sample $U_n\sim \textsf{Uniform}[0,1]$, and add the corresponding vertical edges under $\bbP_{p,q}$ (respectively $\bbP_{p,q'}$) precisely when $U_n \leq q$ (respectively $U_n \leq q'$). In the coupled model, we write~$(\xi^{q}_n(A))_{n \ge 1}$ for the process with parameters~$(p,q)$ and~$(\xi^{q'}_n)_{n\ge 1}$ for the process with parameters~$(p,q')$. In particular,~$r^q_n = \max\xi_n^q(-\bbN)$ and~$r^{q'}_n = \max\xi^{q'}_n(-\bbN)$.

We now set
\[
	\tau = \inf\{ n\in 2\mathbb N: r^{q}_{n} < r_{n}^{q'}\}
\]
	 For all~$m \in 2 \mathbb N$, on the event~$\{ \tau = m\}$ we have
	\begin{equation}
	\label{eq_assumptions_of_durr}
	\xi_m^q(-\bbN) \subset \xi_m^{q'}(-\bbN)  \quad \text{and}\quad \max \xi_m^{q'}(-\bbN) = r_m^{q'} > r_m^q = \max \xi_m^q(-\bbN).
	\end{equation}
	Let~$(\mathcal F_m)_{m \ge 0}$ denote the filtration generated by the percolation configuration: for each~$m$, $\mathcal F_m$ is the~$\sigma$-algebra generated by the percolation configuration (including the uniform random variables) up to (and including) height~$m$. On the event~$\{\tau \le n\}$ we bound
	\begin{align*}
		\bbE [r^{q'}_n - r^q_n \mid \mathcal F_\tau ] &= \bbE[\max \xi_{\tau,n}^{q'}(\xi_\tau^{q'}(-\bbN)) - \max \xi_{\tau,n}^{q}(\xi_\tau^{q}(-\bbN)) \mid \mathcal F_\tau]\\
		&\ge \bbE[\max \xi_{\tau, n}^{q}(\xi_\tau^{q'}(-\bbN)) - \max \xi_{\tau, n}^{q}(\xi_\tau^{q}(-\bbN)) \mid \mathcal F_\tau] \ge 1,
	\end{align*}
	where the last inequality follows from~\eqref{eqn:Durret rnA rnB inequality}, whose assumptions have been verified in~\eqref{eq_assumptions_of_durr}. We have thus proved:
	\[
		\bbE[r_n^{q'} - r_n^q] \ge\bbE[ \bbE[r_n^{q'} - r_n^q \mid \mathcal{F}_\tau] \cdot  \1\{\tau \le n\}] \ge \bbP(\tau \le n).
	\]

	To bound this latter probability, note that at each time $m \in 2\mathbb N$ there is a probability at least $(q'-q)p(1-p)^2$ that the vertical edge leading from $r^{q}_{m}$ to $r^{q}_{m} + (0,2)$ is open under $\bbP_{p,q'}$ but not under $\bbP_{p,q}$, that the site $r^{q}_{m} + (0,2)$ is also open, but that the two sites corresponding to $r^{q}_{m} + (0,1)$ and $r^{q}_{m} + (-1,2)$ are closed, in which case $r^{q}_{m+2} < r^{q'}_{m+2}$. Hence, 
	\begin{equation}
		\label{eq_bound_tau}
		\bbP ( \tau > n) \leq (1-(q'-q)p(1-p)^2)^{\lfloor n/2\rfloor}, 
	\end{equation}
	from which the statement of~\eqref{eq_for_step2}  follows.

For Step \ref{alpha:step:3}, we again follow the strategy of Durrett, take a large integer $M$, set~$\delta = \frac{q}{M}$ and write 
\begin{align*}
\frac{1}{n}\bbE \left[r_{2n}^q - r^0_{2n}\right] = \frac{1}{n}\sum_{m=1}^{Mn} \bbE \left[r_{2n}^{\frac{m\delta}{n}} - r_{2n}^{\frac{(m-1)\delta}{n}}\right] &\geq M\left(1-\left(1-\frac{\delta p(1-p)^2}{n}\right)^n \right).
\end{align*}
Taking $n \to \infty$ and then $M \to \infty$ we get 
\begin{align*}
 \lim_{n \to \infty} \frac{1}{n}\bbE \left[r_{2n}^q - r^0_{2n}\right] \geq \lim_{M \to \infty} M \left(1-\exp \left\{-\frac{q p(1-p)^2}{M}\right\}\right) = qp(1-p)^2,  
\end{align*}
which is the desired statement.

The proof of~\eqref{eq_beta_mon} goes in the exact same way; note in particular that we can get the same expression in the bound~\eqref{eq_bound_tau}.
\end{proof}

\begin{lemma}
\label{lem:alpha:critical}
	If~$p > \pc(0)$, then~$\beta(p,0)^{-1} = \alpha(p,0) \ge 1$.
\end{lemma}
\begin{proof}
We write~$\cC_0 := \{(x,n){\in\bbZ^2}: x \in \xi_n(\{0\})\}$ {for the cluster of the origin.}
Fix $p > \pc(0)$, so that $\mathbb P_{p,0}(|\mathcal C_0| = \infty) > 0$. Throughout this proof, we abbreviate~$\alpha = \alpha(p,0)$ and~$\beta = \beta(p,0)$.

Note that for any $n\in\bbN$
\begin{align*}
\max \xi_n(\{0\}) &{}\le r_n,&\min \xi_n(\{0\}) &{}\ge l_n
\end{align*}
and on the event $|\cC_0|=\infty$, for all $n\in\bbN$,
\begin{equation}
\label{eq_note_cluster}
l_n = \min \xi_n(\{0\}) \le \max \xi_n(\{0\}) = r_n,
\end{equation}
{using the non-crossing property of simple paths.} Taken together with~$\frac{l_{2n}}{2n} \xrightarrow{n \to \infty} \beta$ and~$\frac{r_{2n}}{2n} \xrightarrow{n \to \infty} \alpha$,~\eqref{eq_note_cluster} implies that~$\beta \le \alpha$.

For~$a > 0$, we write
\[V_-(a):= \{(v,n) \in \mathbb Z \times 2 \mathbb N:\; v \le an\},\quad V_+(a):=\{(v,n) \in \mathbb Z \times 2 \mathbb N:\; v \ge an\}.\]
 We claim that
\begin{equation}
\label{eq_note_alpha}
    \alpha = \inf\{a > 0:\; \mathbb P_{p,0}(|\cC_0 \cap V_+(a)| < \infty) = 1\}.
\end{equation}
To see this, first take~$a > \alpha$. Since~$\max \xi_n(\{0\}) \le r_n \; \forall n$ and~$\frac{r_{2n}}{2n} \xrightarrow{n \to \infty} \alpha$, we see that almost surely there are only finitely many~$n \in 2 \mathbb N$ such that~$\max \xi_n(\{0\}) \ge an$, so there are almost surely only finitely many points in~$\cC_0 \cap V_+(a)$. On the other hand, if~$a < \alpha$, we have
\begin{align*}
\mathbb P_{p,0}(|\cC_0 \cap V_+(a)| = \infty) &\ge \mathbb P_{p,0}(\max \xi_n(\{0\}) \ge an \text{ for infinitely many } n \in 2 \mathbb N) \\
    &\ge \mathbb P_{p,0}(|\cC_0| = \infty) > 0,
\end{align*}
where the second inequality follows from~\eqref{eq_note_cluster} and~$\frac{r_{2n}}{2n} \xrightarrow{n \to \infty} \alpha$. This concludes the proof of~\eqref{eq_note_alpha}. Similarly, we have
\begin{equation}
    \label{eq_same_beta}
\beta = \sup\{b > 0:\; \mathbb P_{p,0}(|\cC_0 \cap V_-(b)| < \infty) = 1\}.
\end{equation}
Let~$\Phi: \mathbb R^2 \to \mathbb R^2$ be the reflection about the diagonal~$y=x$, that is,~$\Phi(x,y) = (y,x)$. Since we are taking~$q = 0$, our model has the symmetry~$\cC_0 \stackrel{\text{(law)}}{=} \Phi(\cC_0)$. In particular, for any~$a > 0$,
	\begin{align*}
		\mathbb P_{p,0}( | \cC_0 \cap V_-(1/a) | < \infty) =\mathbb P_{p,0}( |\cC_0 \cap V_+(a)| < \infty).
	\end{align*}
	Together with~\eqref{eq_note_alpha} and~\eqref{eq_same_beta}, this gives~$\beta = 1/\alpha$. We already had~$\beta \le \alpha$, so we obtain~$\beta \le 1$ and~$\alpha \ge 1$.
\end{proof}

\begin{corollary}\label{cor_at_crit}
	We have~$\alpha(\pc(0),0) \ge 1$ and~$\beta(\pc(0),0) \le 1$.
\end{corollary}
\begin{proof}
The function~$p \mapsto \alpha(p,0)$ is non-decreasing, and it is the decreasing limit of the continuous functions~$p \mapsto \inf_{m \le n} (\mathbb E_{p,0}[r_m]/m)$, as~$n \to \infty$. From this, it is easy to deduce that~$p \mapsto \alpha(p,0)$ is right continuous, so it follows from Lemma \ref{lem:alpha:critical} that~$\alpha(\pc(0),0) \ge 1$. An analogous argument applies to~$\beta$.
\end{proof}

\begin{proof}[Proof of Lemma \ref{lem:alpha:strict}]
This follows from combining Lemma \ref{lem_monotone_alpha} and Corollary \ref{cor_at_crit}.
\end{proof}

%%%%%%%%%%%%%%%%%%%%%%%%%%%%%%%%%%%%%%%
%%%%%%%%%%%%%%%%%%%%%%%%%%%%%%%%%%%%%%%
%%%%%%%%%%%%%%%%%%%%%%%%%%%%%%%%%%%%%%%
%%%%%%%%%%%%%%%%%%%%%%%%%%%%%%%%%%%%%%%
%%%%%%%%%%%%%%%%%%%%%%%%%%%%%%%%%%%%%%%
%%%%%%%%%%%%%%%%%%%%%%%%%%%%%%%%%%%%%%%
\subsection{Supercritical box crossing}
\label{app:box}
Our next goal is to prove Lemma \ref{lem:crossing} following \cite{Durrett84}. We start by proving an upper tail bound for the right edge $r_m$.
\begin{lemma}\label{lem_bound_speed}
For any~$p \in (0,1)$,~$q \in [0,1)$ and~$\delta > 0$ there exist~$c > 0$ and~$n_0 \in \bbN$ such that for all~$n \ge n_0$,
\begin{align}\label{eq_max_ineq}
\bbP_{p,q}\left( \exists m \le n:\; r_m   > \alpha(p,q)\cdot  m+\delta n\right) &{}< e^{-cn},\\
\label{eq_max_ineq2}
\bbP_{p,q}\left( \exists m \in \{1,\ldots, n\}:\; \max \xi_{1,m}({-\bbN})  > \alpha(p,q)\cdot m +\delta n\right) &{}< e^{-cn},\\
\nonumber\bbP_{p,q}\left( \exists m \le n:\; l_m < \beta(p,q)\cdot  m - \delta n\right) &{}< e^{-cn},\\
\nonumber\bbP_{p,q}\left(\exists m \in \{1,\ldots, n\}: \;\min \xi_{1,m}(\bbN) < \beta(p,q)\cdot m - \delta n \right)&{}< e^{-cn}.
\end{align}
\end{lemma}
\begin{proof}
We will only prove the first two bounds, as the other two are treated in the same way. Fix~$p,q,\delta$ as in the statement. We abbreviate~$\alpha = \alpha(p,q)$. The desired inequalities are trivial in case~$\alpha = - \infty$, so we assume that~$\alpha \in (-\infty,\infty)$. 

Using the definition of~$\alpha$ in~\eqref{eq_def_of_alpha}, we choose~$M \in 2\bbN$ such that
\begin{equation}
\label{eq_for_choice_M}
\bbE_{p,q}[r_M - \alpha M] < \frac{\delta}{4}M.
\end{equation}
We bound the left-hand side of~\eqref{eq_max_ineq} by
\begin{align}\label{eq_first_term_ldp}
&\bbP_{p,q}\left( \exists k \le \frac{n}{M}:\; r_{Mk}   > \alpha Mk+\frac{\delta}{2} n\right) \\
&\quad + \bbP_{p,q} \left(\exists k \le \frac{n}{M},\; j \in \{0,\ldots, M-1\}: \; r_{Mk+j} - r_{Mk} > {\alpha j+}\frac{\delta}{2}n \right).\label{eq_second_term_ldp}
\end{align}
The probability in~\eqref{eq_first_term_ldp} is smaller than
\begin{align*}
\bbP_{p,q} \left( \exists k \le \frac{n}{M}:\; r_{Mk}   > \alpha Mk + \frac{\delta}{4} M k+\frac{\delta}{4} n\right) \le \bbP_{p,q}\left( \exists k \le \frac{n}{M}:\; \sum_{j=0}^{k} X_j > \frac{\delta}{4}n\right),
\end{align*}
where~$X_1,X_2,\ldots$ are independent random variables, with the distribution of~$r_M - (\alpha+\tfrac{\delta}{4}) M$. These random variables have negative expectation by~\eqref{eq_for_choice_M}. They also have some finite exponential moment; this can be seen using the domination by geometric random variables, as in the proof of Lemma \ref{lem_alpha_beta}. By a large deviation bound (see for instance~\cite[Corollary A.2.7]{Lawler10}), the probability on the right-hand side above is bounded by~$e^{-c_0 n}$, for some~$c_0 > 0$ (depending on~$M$) and~$n$ large enough.

	Next, bounding $\min_{0 \le j \le M-1} (\alpha j + \frac{\delta}{2}n) > \frac{\delta}{4}n$ for~$n$ large, and using the stochastic domination described in the proof of Lemma \ref{lem_alpha_beta},  we bound the probability in~\eqref{eq_second_term_ldp} by
	\[\frac{n}{M}\cdot \bbP_{p,q} \left(  \max_{0 \le j < M} r_j  > \frac{\delta}{4}n \right) \le \frac{n}{M} \cdot \bbP_{p,q} \left(\sum_{j=0}^{M-1} Y_j > \frac{\delta}{4}n \right),\]
	where~$Y_0,\ldots, Y_{M-1}$ are independent~$\mathsf{Geometric}(1-p)$. The right-hand side above is again bounded by~$e^{-c_1 n}$ for some constant~$c_1 > 0$ (depending on~$M$) and~$n$ large enough. This concludes the proof of~\eqref{eq_max_ineq}.

	For~\eqref{eq_max_ineq2}, we first write, for any~$m \ge 2$,
	\begin{align*}
		\max\xi_{1,m}({-\bbN}) &= \max \xi_{2,m}(\xi_{1,2}({-\bbN})) \le \max \xi_{2,m}((-\infty, \max \xi_{1,2}({-\bbN}))).
	\end{align*}
 The right-hand side is stochastically dominated by~$\mathcal Z + r_{m-2}'$, where 
	\[\mathcal Z \stackrel{\text{(distr)}}{=} \max \xi_{1,2}({-\bbN}),\qquad (r_m')_{m \ge 0} \stackrel{\text{(distr)}}{=} (r_m)_{m \ge 0},\]
	and~$\mathcal Z$,~$(r_m')_{m \ge 0}$ are independent. Then, the left-hand side of~\eqref{eq_max_ineq2} is smaller than
	\begin{align*}
		&\bbP_{p,q}\left(\mathcal Z > \frac{\delta}{2}n \right) + \bbP_{p,q}\left(\exists m{\le n} :\; r_{m}' > \alpha \cdot (m+2) + \frac{\delta}{2} n  \right).
	\end{align*}
	The first probability above can be bounded using domination by geometric random variables as before, and the second probability can be bounded using~\eqref{eq_max_ineq}.
\end{proof}

\begin{proof}[Proof of Lemma \ref{lem:crossing}]
	Since the two inequalities are proved in the same way, we will only prove the first. Let~$p$,~$q$,~$\delta$ and~$\varepsilon$ be as in the statement {and write $\alpha=\alpha(p,q)$}.

We let~${R} := (-\tfrac{\delta}{2}n,0) + R(u,v) $, that is,~${R}$ is the parallelogram with vertices
\[(-\tfrac{\delta}{2}n,0),\;(\tfrac{\delta}{2}n,0),\;(-(\tfrac{\delta}{2}+\alpha)n,n),\;((\tfrac{\delta}{2}+\alpha)n,n).\] 
From~$\frac{r_{2n}}{2n} \xrightarrow{n \to \infty} \alpha$, it readily follows that~$\bbP_{p,q}(\mathcal A_n) \xrightarrow{n \to \infty}1$, where
\[\mathcal A_n := \left\{-\frac{\delta}{4}n + \alpha m \le r_m \le \frac{\delta}{4}n + \alpha m \; \text{for all } m\le n,\; m \text{ even} \right\}.\]
On this event, there is a{n open} path~$\gamma = ((x_0,n_0),\ldots,(x_k,n_k))$ such that~$n_0 = 0$,~$x_0 \le 0$,~$n_k = n$, $x_k=r_n\ge (\alpha-\delta/4) n$ and
\[ x_j {\le} \frac{\delta}{4}n + \alpha n_j \text{ for all $j$ for which $n_j$ is even}.\]
If multiple such paths~$\gamma$ exist, we choose one using some arbitrary procedure. In order to prove that~$\gamma$ is entirely contained in~$R'$ with high probability, we only need to prove that the following two situations are unlikely:
\begin{enumerate}
    \item\label{crossing:item:1} $\mathcal A_n$ occurs, but~$x_j \ge \frac{\delta}{2}n + \alpha n_j$ for some~$j$ for which~$n_j$ is odd;
    \item\label{crossing:item:2} $\mathcal A_n$ occurs, but~$x_j \le -\frac{\delta}{2}n + \alpha n_j$ for some~$j$.
\end{enumerate}

The occurrence of \ref{crossing:item:1} would imply~$r_{m+1} - r_m > \frac{\delta}{4}n$ for some~$m \in 2 \bbN$,~$m < n$. To rule this out, we bound this difference by a~$\mathsf{Geometric}(1-p)$ random variable, and use a union bound over the choice of~$m$.

The occurrence of \ref{crossing:item:2} would imply that, for some~$m < n$,
\[\max\xi_{m,n}\left(\left(-\infty,-\frac\delta2n+\alpha m\right]\right)\ge \left(\alpha-\frac\delta4\right)n\]
To rule this out, we use Lemma \ref{lem_bound_speed} and a union bound over the choices of~$m$.	
\end{proof}

%%%%%%%%%%%%%%%%%%%%%%%%%%%%%%%%%%%%%%%
%%%%%%%%%%%%%%%%%%%%%%%%%%%%%%%%%%%%%%%
%%%%%%%%%%%%%%%%%%%%%%%%%%%%%%%%%%%%%%%
%%%%%%%%%%%%%%%%%%%%%%%%%%%%%%%%%%%%%%%
%%%%%%%%%%%%%%%%%%%%%%%%%%%%%%%%%%%%%%%
%%%%%%%%%%%%%%%%%%%%%%%%%%%%%%%%%%%%%%%

\backmatter

%%%%%%%%%%%%%%%%%%%%%%%%%%%%%%%%%%%%
%%%%%%%%%%%%%%%%%%%%%%%%%%%%%%%%%%%%
%%%%%%%%%%%%%%%%%%%%%%%%%%%%%%%%%%%%
%%%%%%%%%%%%%%%%%%%%%%%%%%%%%%%%%%%%
%%%%%%%%%%%%%%%%%%%%%%%%%%%%%%%%%%%%
%%%%%%%%%%%%%%%%%%%%%%%%%%%%%%%%%%%%
%%%%%%%%%%%%%%%%%%%%%%%%%%%%%%%%%%%%
\bmhead{Acknowledgements}

The authors began this project during a workshop at the 
University of Bath, organized by Alexandre Stauffer and DV, 
and funded by a Heilbronn Focused Research Grant. 
We thank 
Aur\'elia Deshayes, 
Janko Gravner, 
Marcelo Hil\'ario, 
Daniel Kious, 
Alexandre Stauffer, 
R\'eka Szab\'o, 
Vincent Tassion, 
Augusto Teixeira 
and Cristina Toninelli 
for enlightening conversations, and the anonymous referee for helpful comments. 
This research was funded in part (IH) by the Austrian Science Fund (FWF) P35428-N 
and in part (EA) by ANR grant ProGraM (ANR-19-CE40-0025). 
BK was supported by a Florence Nightingale Bicentennial Fellowship 
(Oxford Statistics) and a Senior Demyship (Magdalen College). 

%%%%%%%%%%%%%%%%%%%%%%%%%%%%%%%%%%%%
%%%%%%%%%%%%%%%%%%%%%%%%%%%%%%%%%%%%
%%%%%%%%%%%%%%%%%%%%%%%%%%%%%%%%%%%%
%%%%%%%%%%%%%%%%%%%%%%%%%%%%%%%%%%%%
%%%%%%%%%%%%%%%%%%%%%%%%%%%%%%%%%%%%
%%%%%%%%%%%%%%%%%%%%%%%%%%%%%%%%%%%%
%%%%%%%%%%%%%%%%%%%%%%%%%%%%%%%%%%%%
\bibliography{Catalan_v1.2}% common bib file
%% if required, the content of .bbl file can be included here once bbl is generated
%%\input sn-article.bbl

\end{document}